\documentclass{amsart}
\usepackage{a4}

\usepackage{hyperref}
\usepackage[italian, english]{babel}

\usepackage{tikz-cd} 

\theoremstyle{definition}
\newtheorem{definition}{Definition}[section]
\newtheorem{theorem}[definition]{Theorem}
\newtheorem{lemma}[definition]{Lemma}
\newtheorem{proposition}[definition]{Proposition}
\newtheorem{remark}[definition]{Remark}

\usepackage{xcolor} 
\usepackage[font={small,it}]{caption}

\usepackage{enumerate} 

\begin{document}

\title[SID of curve sections of the known EF 3-folds]{Simple isotropic decompositions of the curve sections of the known Enriques-Fano threefolds}
\author[V. Martello]{Vincenzo Martello}
\address{Dipartimento di Matematica, Universit\`{a} della Calabria, Arcavacata di Rende (CS)}
\email{vincenzomartello93@gmail.com}

\thanks{}
\date{}
\dedicatory{}
\commby{}

\begin{abstract}
In this paper, we describe the simple isotropic decompositions of the curve sections of the known Enriques-Fano threefolds. The simple isotropic decompositions allow us to identify the irreducible components of the moduli space of the polarized Enriques surfaces. Thus, our analysis will enable us to show to which families of polarized Enriques surfaces the hyperplane sections of the Enriques-Fano threefolds belong.
\end{abstract}

\maketitle
\section{Introduction}

Let $\mathcal{E}$ be the smooth 
$10$-dimensional moduli space parametrizing the Enriques surfaces.
A \textit{polarized Enriques surface} is a pair made of an Enriques surface together with the linear equivalence class of an ample divisor on it. 
For fixed integers $g>1$ and $\phi>0$, let $\mathcal{E}_{g,\phi}$ be the moduli space of the 
polarized Enriques surfaces $(S,H)$
satisfying $H^2=2g-2$ and $\phi(H) =\phi$, where
$$\phi (H) :=\min\{E\cdot H | E\in \operatorname{NS}(S), E^2=0, E>0\}.$$
Although the space $\mathcal{E}$ is irreducible, the space $\mathcal{E}_{g,\phi}$ is in general reducible. 
The irreducible components of $\mathcal{E}_{g,\phi}$ can be described by decompositions of $H$ into effective sums of divisors having zero self-intersection, called \textit{SID} of $H$, which we will define in \S~\ref{subsec:preliminarSid}.
If $g \leq 20$ or $\phi \leq 4$, this is studied in \cite[Corollaries 1.3-1.4]{CDGK19}. The general case (for any $g$ and $\phi$) is studied in \cite{Knu20}, where the author introduced the notion of \textit{fundamental presentation}, which is a particular SID whose coefficients, called \textit{fundamental coefficients}, uniquely determine the irreducible component of the moduli space $\mathcal{E}_{g,\phi}$ (see \cite[Theorem 5.9]{Knu20}); furthermore, alternatively to the fundamental coefficients, the irreducible components of the moduli space of polarized Enriques surfaces are determined by a $\phi$\textit{-vector}, which generalizes the $\phi$-value (see \cite[Theorem 1.3-1.4]{Knu20}). 

Let $(W, \mathcal{L})$ be an \textit{Enriques-Fano threefold}, that is a pair made of a normal threefold $W$ together with a complete linear system $\mathcal{L}$ of ample Cartier divisors whose general element $S\in \mathcal{L}$ is an Enriques surface, and such that $W$ is not a 
\textit{generalized cone} over $S$, i.e., $W$ is not obtained by contraction of the negative section on the $\mathbb{P}^1$-bundle $\mathbb{P}(\mathcal{O}_{S}\oplus \mathcal{O}_{S}(S))$ over $S$. The linear system $\mathcal{L}$ defines a rational map $\phi_{\mathcal{L}} : W \dashrightarrow \mathbb{P}^{p}$, where $p:=\frac{S^3}{2}+1=\dim \mathcal{L}$ is called the \textit{genus} of $W$. It is known that $2\le p\le 17$ and that the bound is sharp (see \cite{KLM11} and \cite{Pro07}). 
Though improperly, we will refer to the elements of $\mathcal{L}$ as \textit{hyperplane sections} of $W$ and to the curve intersections of two elements of $\mathcal{L}$ as \textit{curve sections} of $W$.  
As a matter of fact, some authors define an Enriques-Fano threefold just as a non-degenerate threefold $W'\subset \mathbb{P}^{N}$ whose general hyperplane section $S'$ is an Enriques surface, and such that $W'$ is not a cone over $S'$ (see for example \cite[Definition 1.3]{KLM11}). 
The latter case, indeed, just requires us to take its \textit{normalization} $\nu : W \to W'$ to obtain an Enriques-Fano threefold in the general sense, that is $(W, \mathcal{L}:=|\mathcal{O}_{W}(\nu^* S')|)$. 
Nevertheless, the classification of Enriques-Fano threefolds still remains an open question:
in \S~\ref{sec:list} we will list the known examples and their properties. 

Let us denote by $H$ the class of a curve section on a smooth hyperplane section $S\in \mathcal{L}$ of a known Enriques-Fano threefold $(W,\mathcal{L})$ of genus $p$.
The main result of this paper is Theorem~\ref{thm:SID}, which describes the SID of $H$ and the value $\phi (H)$: thus, our analysis allows to determine to which irreducible components of $\mathcal{E}_{p,\phi}$ the hyperplane sections of the Enriques-Fano threefolds belong (see Theorem~\ref{thm:SID}). 

The idea of the proof of Theorem~\ref{thm:SID} is the following: by taking the inequality $\phi^2 \le H^2 = 2p-2$ (see \cite[Cor. 2.7.1]{CoDo89}) and by studying the nature of the map $\phi_{\mathcal{L}}: W \dashrightarrow \mathbb{P}^{p}$ (thanks to \cite[Lemma 4.1]{CDGK20} and \cite{CoDo89}[Thm 4.4.1, Prop. 4.5.1, Thm 4.6.1]), one finds all the possible values for $\phi$ and all the possible SID of $H$ (see \cite[Appendix]{CDGK19}); subsequently, by studying specific properties of $H$ (such as its divisibility in $\operatorname{Pic}(S)$) one deduces the only possible SID.

In order to find useful information to prove Theorem~\ref{thm:SID}, in \S~\ref{subsec:Fano13} we will study the curve sections of the Enriques-Fano threefolds of genus $13$ found by Fano in \cite[\S 8]{Fa38}; in \S~\ref{subsec:Fano9} we will study the curve sections of the Enriques-Fano threefolds of genus $9$ found by Fano in \cite[\S 7]{Fa38}; in \S~\ref{subsec:pro13} we will deepen the study of the Enriques-Fano threefold of genus $13$ mentioned very briefly by Prokhorov in \cite[Remark 3.3]{Pro07}: furthermore, we will analyze its curve sections. We will work over the field $\mathbb{C}$ of the complex numbers.

\subsection*{Acknowledgment}
The results of this paper are contained in my PhD-thesis. I would like to thank my main advisors C. Ciliberto and C. Galati and my co-advisor A.L. Knutsen for our stimulating conversations and for providing me very useful suggestions. 
I would also like to acknowledge PhD-funding from the Department of Mathematics and Computer Science of the University of Calabria. 

\section{Preliminary results}\label{subsec:preliminarSid}

In this section we want to collect known results about Enriques surfaces.
Let $S$ be an Enriques surface. Any irreducible curve $C$ on $S$ satisfies $C^2=2p_a(C)-2\ge -2$, with equality occurring if and only if $C\cong \mathbb{P}^1$. If $S$ contains such a curve, it is called \textit{nodal}, otherwise it is said to be \textit{unnodal}. The general Enriques surface is unnodal (see \cite{Dolg16}). A divisor $E$ on $S$ is said to be \textit{isotropic} if $E^2 = 0$ and $E$ is not numerically equivalent to $0$, and it is said to be \textit{primitive} if it is non-divisible in $\operatorname{Num}(S)$.
On an unnodal Enriques surface, any effective primitive isotropic divisor $E$ is represented by an irreducible curve of arithmetic genus one.
Let $(S,H)$ be a polarized Enriques surface.
It is known that there are $10$ primitive effective isotropic divisors $E_1, \dots ,E_{10}$ such that $E_i \cdot E_j = 1$ for $i \ne j$ and such that
\begin{equation}
\label{eq:SID}
H \sim a_0E_{1,2}+a_1E_1+\dots +a_{10}E_{10} + \epsilon K_S
\end{equation}
where $E_{1,2}\sim \frac{1}{3}(E_1+\dots +E_{10})-E_1-E_2$ and $a_0,a_1,\dots a_{10}$ are nonnegative integers with
$$\begin{cases}
\text{either }a_0 = 0 \text{ and } \#\{i | i \in \{1, \dots , 10\}, a_i >0 \} \ne 9,\\
\text{or }a_{10} = 0,
\end{cases}$$
and
$$ \epsilon = \begin{cases}
0,\text{ if }H+K_S\text{ is not } 2\text{-divisible on }\operatorname{Pic}(S),\\
1,\text{ if }H+K_S\text{ is } 2\text{-divisible on }\operatorname{Pic}(S),
\end{cases}
$$
(see \cite[Corollary 4.7]{CDGK19}). Let us observe that $E_{1,2}\cdot E_1 = E_{1,2}\cdot E_2 = 2$ and $E_{1,2}\cdot E_i=1$ for $i=3,\dots , 10$.
The expression (\ref{eq:SID}) is called \textit{simple isotropic decomposition} (simply, \textit{SID}) of $H$.

Let us take 
an Enriques-Fano threefold $(W, \mathcal{L})$ of genus $p$ and let us denote by $H$ the class of a curve section of $W$ on a general (smooth) hyperplane section $S$. Hence we have $|H| = \mathcal{L} |_{S}$ (see \cite[Lemma 4.1 (i)]{CDGK20}). We set $\phi := \phi (H)$ and we recall that $\phi^2 \le H^2 = 2p-2$ (see \cite[Cor. 2.7.1]{CoDo89}). 
We say that the rational map associated with $|H|$ is \textit{hyperelliptic} if $p=2$ or if it is of degree $2$ onto a surface of degree $p-2$ in $\mathbb{P}^{p-1}$; we say that it is \textit{superelliptic} if $p=2$ or if it is of degree $2$ onto a surface of degree $p-1$ in $\mathbb{P}^{p-1}$ (see \cite[p. 229]{CoDo89}). 
We have the following results, which we will use later:
\begin{itemize}
\item[(a)] by \cite[Proposition 4.5.1]{CoDo89}
$$\phi = 1 \Leftrightarrow |H| \text{ has 2 simple base points } \Leftrightarrow \phi_{\mathcal{L}} \text{ is hyperelliptic on }S;$$
\item[(b)] by \cite[Lemma 4.1]{CDGK20} and \cite[Theorem 4.4.1]{CoDo89}
$$\phi \ge 2 \Leftrightarrow |H| \text{ base point free } \Leftrightarrow \mathcal{L} \text{ base point free};$$
\item[(c)] by \cite[Lemma 4.1 (i)]{CDGK20} and \cite[Theorems 4.4.1, 4.6.1]{CoDo89} (since $H$ ample)
$$\phi \ge 3 \Leftrightarrow \phi_{\mathcal{L}} \text{ is an isomorphism on }S.$$
\end{itemize}
In the last case (c) we get that $\phi_{\mathcal{L}}(W)\subset \mathbb{P}^p$ is a 
threefold whose general hyperplane section is a smooth Enriques surface.

For later reference, Table~\ref{tab:Ep-phi} will provide with all the irreducible components of $\mathcal{E}_{p,\phi}$ for $2\le p \le 10$ and $p=13,17$, as well as the corresponding simple isotropic decompositions (see \cite[Appendix]{CDGK19}).

\begin{table}[!ht]\small
\begin{tabular}{|c|c|c|c|c|c|c|c|c|} 
\cline{1-4}\cline{6-9}
$p$ &  $\phi$ & comp. & SID &  & $p$ & $\phi$ & comp. & SID\\
\cline{1-4}\cline{6-9}
$2$ &  $1$ & $\mathcal{E}_{2,1}$ & $E_1+E_2$ & & $10$ & $1$ & $\mathcal{E}_{10,1}$ & $9E_1+E_2$\\ 
$3$ &  $1$ & $\mathcal{E}_{3,1}$ & $2E_1+E_2$ & & $10$ & $2$ & $\mathcal{E}_{10,2}$ & $4E_1+E_2+E_3$\\
$3$ &  $2$ & $\mathcal{E}_{3,2}$ & $E_1+E_{1,2}$ & & $10$ & $3$ & $\mathcal{E}_{10,3}^{(I)}$ & $2E_1+E_2+E_3+E_4$\\
$4$ &  $1$ & $\mathcal{E}_{4,1}$ & $3E_1+E_2$ & & $10$ & $3$ & $\mathcal{E}_{10,3}^{(II)}$ & $3(E_1+E_2)$\\
$4$ &  $2$ & $\mathcal{E}_{4,2}$ & $E_1+E_2+E_3$ & & $10$ & $4$ & $\mathcal{E}_{10,4}$ & $2E_{1,2}+E_1+E_2$\\
$5$ &  $1$ & $\mathcal{E}_{5,1}$ & $4E_1+E_2$ & & $13$ & $1$ & $\mathcal{E}_{13,1}$ & $12E_1+E_2$\\
$5$ &  $2$ & $\mathcal{E}_{5,2}^{(I)}$ & $2E_1+E_{1,2}$ & & $13$ & $2$ & $\mathcal{E}_{13,2}^{(I)}$ & $6E_1+E_{1,2}$\\
$5$ &  $2$ & $\mathcal{E}_{5,2}^{(II)+}$ & $2(E_1+E_{2})$ & & $13$ & $2$ & $\mathcal{E}_{13,2}^{(II)+}$ & $2(3E_1+E_2)$\\
$5$ &  $2$ & $\mathcal{E}_{5,2}^{(II)-}$ & $2(E_1+E_{2})+K_S$ & & $13$ & $2$ & $\mathcal{E}_{13,2}^{(II)-}$ & $2(3E_1+E_2)+K_S$\\
$6$ &  $1$ & $\mathcal{E}_{6,1}$ & $5E_1+E_2$ & & $13$ & $3$ & $\mathcal{E}_{13,3}^{(I)}$ & $3E_1+E_2+E_3+E_4$\\
$6$ &  $2$ & $\mathcal{E}_{6,2}$ & $2E_1+E_2+E_{3}$ & & $13$ & $3$ & $\mathcal{E}_{13,3}^{(II)}$ & $4E_1+3E_2$\\
$6$ &  $3$ & $\mathcal{E}_{6,3}$ & $E_1+E_2+E_{1,2}$ & & $13$ & $4$ & $\mathcal{E}_{13,4}^{(I)}$ & $2E_1+2E_2+E_{1,2}$\\
$7$ &  $1$ & $\mathcal{E}_{7,1}$ & $6E_1+E_2$ & & $13$ & $4$ & $\mathcal{E}_{13,4}^{(II)+}$ & $2(E_1+E_2+E_3)$\\
$7$ &  $2$ & $\mathcal{E}_{7,2}^{(I)}$ & $3E_1+E_{1,2}$ & & $13$ & $4$ & $\mathcal{E}_{13,4}^{(II)-}$ & $2(E_1+E_2+E_3)+K_S$\\
$7$ &  $2$ & $\mathcal{E}_{7,2}^{(II)}$ & $3E_1+2E_2$ & & $13$ & $4$ & $\mathcal{E}_{13,4}^{(III)}$ & $3E_1+2E_{1,2}$\\
$7$ &  $3$ & $\mathcal{E}_{7,3}$ & $E_1+E_2+E_3+E_4$ & & $17$ & $1$ & $\mathcal{E}_{17,1}$ & $16E_1+E_2$\\
$8$ &  $1$ & $\mathcal{E}_{8,1}$ & $7E_1+E_2$ & & $17$ & $2$ & $\mathcal{E}_{17,2}^{(I)}$ & $8E_1+E_{1,2}$\\
$8$ &  $2$ & $\mathcal{E}_{8,2}$ & $3E_1+E_2+E_3$ & & $17$ & $2$ & $\mathcal{E}_{17,2}^{(II)+}$ & $2(4E_1+E_2)$\\
$8$ &  $3$ & $\mathcal{E}_{8,3}$ & $2E_1+E_3+E_{1,2}$ & & $17$ & $2$ & $\mathcal{E}_{17,2}^{(II)-}$ & $2(4E_1+E_2)+K_S$\\
$9$ &  $1$ & $\mathcal{E}_{9,1}$ & $8E_1+E_2$ & & $17$ & $3$ & $\mathcal{E}_{17,3}$ & $5E_1+E_3+E_{1,2}$\\
$9$ &  $2$ & $\mathcal{E}_{9,2}^{(I)}$ & $4E_1+E_{1,2}$ & & $17$ & $4$ & $\mathcal{E}_{17,4}^{(I)}$ & $3E_1+2E_2+2E_3$\\
$9$ &  $2$ & $\mathcal{E}_{9,2}^{(II)+}$ & $2(2E_1+E_2)$ & & $17$ & $4$ & $\mathcal{E}_{17,4}^{(II)}$ & $3E_1+2E_2+E_{1,2}$\\
$9$ &  $2$ & $\mathcal{E}_{9,2}^{(II)-}$ & $2(2E_1+E_2)+K_S$ & & $17$ & $4$ & $\mathcal{E}_{17,4}^{(III)+}$ & $2(2E_1+E_{1,2})$\\
$9$ &  $3$ & $\mathcal{E}_{9,3}^{(I)}$ & $2E_1+E_2+E_{1,2}$ & & $17$ & $4$ & $\mathcal{E}_{17,4}^{(III)-}$ & $2(2E_1+E_{1,2})+K_S$\\
$9$ &  $3$ & $\mathcal{E}_{9,3}^{(II)}$ & $2E_1+2E_2+E_3$ & & $17$ & $4$ & $\mathcal{E}_{17,4}^{(IV)+}$ & $4(E_1+E_2)$\\
$9$ &  $4$ & $\mathcal{E}_{9,4}^{+}$ & $2(E_1+E_{1,2})$ & & $17$ & $4$ & $\mathcal{E}_{17,4}^{(IV)-}$ & $4(E_1+E_2)+K_S$\\
$9$ &  $4$ & $\mathcal{E}_{9,4}^{-}$ & $2(E_1+E_{1,2})+K_S$ & & $17$ & $5$ & $\mathcal{E}_{17,5}$ & $2E_1+E_3+E_4+E_5+E_{1,2}$\\
\cline{1-4}\cline{6-9} 
\end{tabular} 
\caption{All irreducible components of $\mathcal{E}_{p,\phi}$ for $2\le p \le 10$ and $p=13,17$.} 
\label{tab:Ep-phi}   
\end{table}

\begin{remark}
A projective variety $X\subset \mathbb{P}^N$ is said to be $k$-\textit{extendable} if there exists a projective variety $V \subset \mathbb{P}^{N+k}$, that is not a cone, such that $X = V \cap \mathbb{P}^N$ (transversely) and $\dim V = \dim X +k$.
It is known that 
if $S\subset \mathbb{P}^{N}$ is an unnodal Enriques surface which is $1$-extendable, then $(S, \mathcal{O}_{S}(1))$ belongs to the following list:
$\mathcal{E}^{(IV)^+}_{17,4}$, $\mathcal{E}^{(II)^+}_{13,4}$, $\mathcal{E}^{(II)}_{13,3}$, $\mathcal{E}^{(II)}_{10,3}$, $\mathcal{E}^{+}_{9,4}$, $\mathcal{E}^{(II)}_{9,3}$, $\mathcal{E}_{7,3}$ (see \cite[Corollary 1.2]{CDGK20}).
\end{remark}

\section{Known Enriques-Fano threefolds}\label{sec:list}

Examples of Enriques-Fano threefolds have been found by several authors: in this section we will list them all while giving some notation.

\begin{definition}\label{def:terminal canonical sing}
Let $W$ be a normal variety such that its canonical divisor $K_{W}$ is $\mathbb{Q}$-Cartier.
Let 
$f : \widetilde{W} \to W$ be a resolution of the singularities of $W$ and let $\{E_{i}\}_{i\in I}$ be the family of all irreducible exceptional divisors. Since 
$K_{\widetilde{W}} = f^{*}\left(K_{W}\right) + \sum_{i\in I} a_{i} E_{i}$
with $a_i \in \mathbb{Q},$
we say that the singularities of $W$ are \textit{terminal} if $a_{i} > 0$ for all $i$, and 
they are \textit{canonical} if $a_{i} \ge 0$ for all $i$.
\end{definition}

It is known that any Enriques-Fano threefold $(W,\mathcal{L})$ is singular with isolated 
singularities (see \cite[Lemma 3.2]{CoMu85}): moreover 
$K_{W}$ is $2$-Cartier and the singularities are canonical (see \cite{Ch96}). Furthermore, if $f : \widetilde{W} \to W$ is a desingularization of $W$, then $h^1(\widetilde{W},\mathcal{O}_{\widetilde{W}})=0$ (see \cite[Lemma 4.1]{CDGK20}). 

\begin{definition}\label{def:Reye}
Let $\mathcal{R}$ be a $3$-dimensional linear system of quadric sufaces of $\mathbb{P}^3$. Let us suppose that $\mathcal{R}$ is sufficiently general, that is: 
\begin{itemize}
\item[(i)]$\mathcal{R}$ is base point free;
\item[(ii)] if $l$ is a double line for $Q \in \mathcal{R}$, then $Q$ is the unique quadric surface in $\mathcal{R}$ containing $l$. 
\end{itemize}
A \textit{Reye congruence} is a surface obtained as the set
$$\{ l \in \mathbb{G}(1,3) | l \text{ is contained in a pencil contained in } \mathcal{R} \},$$ 
where $\mathbb{G}(1,3)$ denotes the Grassmannian variety of lines in $\mathbb{P}^3$. 
\end{definition}


The first to deal with the problem of classification of Enriques-Fano threefolds was Fano (see \cite{Fa38}), who found the following ones, which we will call \textit{F-EF 3-folds}:
\begin{itemize}
\item[(i)] the Enriques-Fano threefold $W_{F}^{6}\subset \mathbb{P}^{6}$ of genus $p=6$ given by the image of $\mathbb{P}^{3}$ via the rational map defined by the linear system $\mathcal{P}$ of the septic surfaces with double points along three twisted cubics having five points in common (see \cite[\S 3]{Fa38}):
such a threefold is \textit{rational}
and its
hyperplane sections 
are Reye congruences (see also \cite[Proposition 3]{Co83});
\item[(ii)] the Enriques-Fano threefold $W_{F}^{7}\subset \mathbb{P}^{7}$ of genus $p=7$ given by the image of $\mathbb{P}^{3}$ via the rational map defined by the linear system $\mathcal{X}$ of the sextic surfaces having double points along the six edges of a tetrahedron and containing a plane cubic curve intersecting each edge at one point (see \cite[\S 4]{Fa38}):
such a threefold is \textit{rational};
\item[(iii)] the Enriques-Fano threefold $W_{F}^{9}\subset \mathbb{P}^{9}$ of genus $p=9$ given by the image of $\mathbb{P}^{3}$ via the rational map defined by the linear system $\mathcal{K}$ of the septic surfaces having double points along the six edges of two trihedra (see \cite[\S 7]{Fa38}):
such a threefold is \textit{rational};
\item[(iv)] the Enriques-Fano threefold $W_{F}^{13}\subset \mathbb{P}^{13}$ of genus $p=13$ given by the image of $\mathbb{P}^{3}$ via the rational map defined by the linear system $\mathcal{S}$ of the sextic surfaces having double points along the six edges of a tetrahedron (see \cite[\S 8]{Fa38}):
such a threefold is \textit{rational};
\end{itemize}

and one ``exceptional'' case (according to Fano's construction): 
\begin{itemize}
\item[(0)] the famous \textit{Enriques threefold} $W^{4}_{F}\subset \mathbb{P}^{4}$, which is a singular sextic hypersurface whose hyperplane section is a sextic surface in $\mathbb{P}^{3}$ with double points along the six edges of a tetrahedron (see \cite[\S 10]{Fa38}):
\begin{itemize}
\item it has equation
\begin{center}
$x_1x_2x_3x_4(x_0^2+x_0\sum_{i=1}^4a_ix_i+\sum_{i,j=1}^{4}b_{ij}x_ix_j)+$

$+c_1x_2^2x_3^2x_4^2+c_2x_1^2x_3^2x_4^2+c_3x_1^2x_2^2x_4^2+c_4x_1^2x_2^2x_3^2 = 0,$
\end{center}
where $x_0,\dots ,x_4$ are the homogeneous coordinates of $\mathbb{P}^4$, and $a_i$, $b_{ij}$ and $c_i$ are sufficiently general complex numbers;
\item it has double points along six planes, which are given by the intersections of four $3$-dimensional projective spaces two by two, and which all pass through the same point;
\item the general Enriques threefolds $W_F^4$ have been proved to be \textit{non-rational} by Picco-Botta and Verra in \cite{P-BV83};
\item it is also contained in \cite{Ba94} and \cite{Sa95} (see (\ref{bayle4fano}) below). 
\end{itemize}
\end{itemize}
Furthermore, as noted by Conte in \cite[p. 225]{Co82}, there is also another ``exceptional'' case in Fano's paper:
\begin{itemize}
\item[(00)] the threefold $W_F^3$ given by a quadruple $\mathbb{P}^3$ (see \cite[\S 2]{Fa38}):
\begin{itemize}
\item it is worth mentioning it, because it is also contained in \cite{Ba94} and \cite{Sa95} (see (\ref{bayle3quadruple}) below).
\end{itemize}
\end{itemize}

The arguments of Fano 
are not free from gaps: indeed, Conte and Murre proved, under certain assumptions, several results that Fano had only stated (see \cite{CoMu85}). However, Conte and Murre did not address the classification problem. Under the assumption that the singularities are terminal cyclic quotients, Enriques-Fano threefolds were classified by Bayle in \cite{Ba94} (and in a similar and independent way by Sano in \cite{Sa95}). 
Let us see how.
Bayle assumed the following fact:
let $(W,\mathcal{L})$ be an Enriques-Fano threefold such that $W$ is the quotient $X/ \sigma$ of a smooth Fano threefold $X$ where $\sigma$ is an involution of $X$ with finitely many fixed points. 
The number of 
these
fixed points 
must be $8$ (see \cite[\S 4.1]{Ba94}). Furthermore, the images of these $8$ points 
via the quotient map $\pi : X \to W$ are eight singular points of $W$ whose tangent cone is a cone over a Veronese surface (see \cite[\S 3]{Ba94}).
By Bayle's assumption we also have 
\begin{itemize}
\item[(i)] $b_{2}(X)+\frac{b_{3}(X)}{2} \equiv 1 \pmod{2}$, where $b_{i}(X):= \operatorname{rank} H_{i}(X,\mathbb{R})$ is the $i^{th}$ Betti's number of $X$ (see \cite[\S 4.2]{Ba94});
\item[(ii)] $\deg X = (-K_{X})^{3} = 4p-4$ is divisible by $4$ (see \cite[\S 4.3]{Ba94}).
\end{itemize}
Bayle's approach to the classification is as follows. He considered all the smooth Fano threefolds, classified by Iskovskih in \cite{Is77} and \cite{Is78} and by Mori and Mukai in \cite{MoMu}, and he eliminated the ones that do not satisfy the afore-mentioned two properties: though a Fano threefold has been erroneously omitted by Mori and Mukai (see \cite{MoMu03}), this has no consequence for Bayle's work, since the degree of this threefold is not divisible by $4$. By studying the remaining smooth Fano threefolds, Bayle found that only $14$ of them have an involution with $8$ fixed points: thus, he found fourteen Enriques-Fano threefolds, by constructing the quotient map $\pi : X \to W$ as the map defined by the sublinear system of $|-K_X|$ given by the $\sigma$-invariant elements. We will refer to them as \textit{BS-EF 3-folds}. They are:

\begin{enumerate}[(I)]
\item\label{bayle2} the Enriques-Fano threefold $W_{BS}^2$ of genus $p=2$ given by the quotient 
of a double cover of a smooth quadric hypersurface of $\mathbb{P}^{4}$ branched in an optic surface (see \cite[\S 6.1.6]{Ba94}): 
\begin{itemize}
\item[--] according to \cite[p. 23]{Ba94}, these $W_{BS}^2$ depend on $25$ moduli;
\item[--] these $W_{BS}^2$ can be also obtained as quotient of the complete intersection of a quadric and quartic in $\mathbb{P}(1^5; 2)$;
\item[--] it is also found by Sano (see \cite[Theorem 1.1 No.1]{Sa95});
\item[--] Cheltsov \textit{conjectured} that $W_{BS}^2$ is \textit{non-rational} (see \cite[Conjecture 19]{Ch04}); 
\end{itemize}

\item\label{bayle3quadruple} the Enriques-Fano threefold $W_{BS}^3$ of genus $p=3$ given by the quotient 
of the complete intersection of three quadric hypersurfaces of $\mathbb{P}^{6}$ (see \cite[\S 6.1.5]{Ba94}): 
\begin{itemize}
\item[--] 
the linear system on $W_{BS}^3$, whose general element is an Enriques surface, defines a morphism and a quadruple cover of $\mathbb{P}^3$;
\item[--] according to \cite[p. 22]{Ba94}, the number of moduli of $W_{BS}^3$ is $15$;
\item[--] it is also found by Sano (see \cite[Theorem 1.1 No.2]{Sa95});
\item[--] Cheltsov \textit{conjectured} that $W_{BS}^3$ is \textit{non-rational} (\cite[Conjecture 19]{Ch04}); 
\end{itemize}

\item\label{bayle3double} the Enriques-Fano threefold $\overline{W}_{BS}^{3}$ of genus $p=3$ given by the quotient 
of the blow-up of $B_2$ along a curve given by the intersection of two elements of $|-\frac{1}{2}K_{B_{2}}|$, where $B_{2}$ is the double cover of $\mathbb{P}^{3}$ branched in a smooth quartic surface (see \cite[\S 6.2.7]{Ba94}): 
\begin{itemize}
\item[--] according to \cite[p. 34]{Ba94}, the number of moduli of $\overline{W}_{BS}^3$ is $15$;
\item[--] these $\overline{W}_{BS}^3$ can also be obtained as quotient of the blow-up of a smooth quartic hypersurface of $\mathbb{P}(1^4; 2)$, along a smooth elliptic curve, which is cut out by two hypersurfaces of degree one;
\item[--] it is also found by Sano (see \cite[Theorem 1.1 No.3]{Sa95});
\end{itemize}

\item\label{bayle4fano} the Enriques-Fano threefold $W_{BS}^4$ of genus $p=4$ given by the quotient 
of a double cover of $\mathbb{P}^{1}\times \mathbb{P}^{1} \times \mathbb{P}^{1}$ branched in a divisor of multidegree $(2,2,2)$ (see \cite[\S 6.3.3]{Ba94}): 
\begin{itemize}
\item[--] 
the linear system on $W_{BS}^4$, whose general element is an Enriques surface, defines a rational map birational onto the image, which is the Enriques threefold $W_F^4\subset \mathbb{P}^4$;
\item[--] according to \cite[p. 40]{Ba94}, the number of moduli of $W_{BS}^4$ is $10$;
\item[--] it is also found by Sano (see \cite[Theorem 1.1 No.5]{Sa95});
\item[--] it is \textit{non-rational} (see \cite{P-BV83});
\end{itemize}

\item\label{bayle4} the Enriques-Fano threefold $\overline{W}_{BS}^{4}$ of genus $p=4$ given by the quotient
of $\mathbb{P}^{1}\times S_{2}$ where $S_{2}$ is a double cover of $\mathbb{P}^{2}$ branched in a quartic curve (see \cite[\S 6.6.2]{Ba94}): 
\begin{itemize}
\item[--] according to \cite[p. 61]{Ba94}, the number of moduli of $\overline{W}_{BS}^4$ is $4$;
\item[--] it is also found by Sano (see \cite[Theorem 1.1 No.4]{Sa95});
\item[--] it is \textit{rational} (see \cite[Remark 7.3]{Ch97});
\end{itemize}

\item\label{bayle5birational} the Enriques-Fano threefold $W_{BS}^{5}$ of genus $p=5$ given by the quotient 
of the blow-up of a smooth intersection of two quadric hypersurfaces of $\mathbb{P}^{5}$, along the elliptic curve given by the intersection of two hyperplane sections (see \cite[\S 6.2.2]{Ba94}): 
\begin{itemize}
\item[--] according to \cite[p. 25]{Ba94}, the number of moduli of these threefolds is $7$;
\item[--] it is also found by Sano (see \cite[Theorem 1.1 No.7]{Sa95});
\item[--] it was accidentally not listed in \cite[Theorem B]{Ba94};
\item[--] it is \textit{rational} (see \cite[Remark 7.3]{Ch97});
\end{itemize}

\item\label{bayle5double} the Enriques-Fano threefold $\overline{W}_{BS}^{5}$ of genus $p=5$ given by the quotient 
of a double cover of $\mathbb{P}^3$, branched in a smooth quartic surface (see \cite[\S 6.1.2]{Ba94}):
\begin{itemize}
\item[--] according to \cite[p. 18]{Ba94}, the number of moduli of these threefolds is $11$;
\item[--] these $\overline{W}_{BS}^5$ can be also obtained as quotient of a quartic hypersurface of $\mathbb{P}(1^4; 2)$;
\item[--] it is also found by Sano (see \cite[Theorem 1.1 No.8]{Sa95});
\item[--] it is \textit{rational} (see \cite[Theorem 1]{Ch04});
\end{itemize}

\item\label{bayle6fano} the Enriques-Fano threefold $W_{BS}^{6}$ of genus $p=6$ given by the quotient 
of the complete intersection of three divisors of bidegree $(1,1)$ on $\mathbb{P}^{3} \times \mathbb{P}^{3}$ (see \cite[\S 6.2.4]{Ba94}):
\begin{itemize}
\item[--] according to \cite[p. 29]{Ba94}, the number of moduli of these threefolds is $24$;
\item[--] it is also found by Sano (see \cite[Theorem 1.1 No.9]{Sa95});
\item[--] it is \textit{rational} (see \cite[Corollary 7.2]{Ch97});
\end{itemize}

\item\label{bayle7} the Enriques-Fano threefold $\overline{W}_{BS}^7$ of genus $p=7$ given by the quotient 
of $\mathbb{P}^{1} \times S_{4}$, where $S_{4}$ is a Del Pezzo surface of degree $4$ in $\mathbb{P}^{4}$ (see \cite[\S 6.6.1]{Ba94}): 
\begin{itemize}
\item[--] according to \cite[p. 59]{Ba94}, the number of moduli of these threefolds is $2$;
\item[--] it is also found by Sano (see \cite[Theorem 1.1 No.10]{Sa95});
\item[--] it is \textit{rational} (see \cite[Corollary 7.2]{Ch97});
\end{itemize}

\item\label{bayle7fano} the Enriques-Fano threefold $W_{BS}^{7}$ of genus $p=7$ given by the quotient 
of a smooth divisor on $\mathbb{P}^{1} \times \mathbb{P}^{1} \times \mathbb{P}^{1} \times \mathbb{P}^{1}$ of multidegree $(1,1,1,1)$ (see \cite[\S 6.4.1]{Ba94}):
\begin{itemize}
\item[--] according to \cite[p. 46]{Ba94}, the number of moduli of these threefolds is $3$;
\item[--] it is also found by Sano (see \cite[Theorem 1.1 No.11]{Sa95});
\item[--] it is \textit{rational} (see \cite[Corollary 7.2]{Ch97});
\end{itemize}

\item\label{bayle8} the Enriques-Fano threefold $W_{BS}^{8}$ of genus $p=8$ given by the quotient 
of the blow-up of the cone over a quadric surface $Q\subset \mathbb{P}^{3}$ along the disjoint union of the vertex and an elliptic curve on $Q$ (see \cite[\S 6.4.2]{Ba94}):
\begin{itemize}
\item[--] according to \cite[p. 51]{Ba94}, the number of moduli of these threefolds is $2$;
\item[--] Sano erroneously omits it (see \cite[p. 378]{Sa95});
\item[--] it is \textit{rational} (see \cite[Corollary 7.2]{Ch97});
\end{itemize} 
 
\item\label{bayle9fano} the Enriques-Fano threefold $W_{BS}^{9}$ of genus $p=9$ given by the quotient 
of the intersection of two quadrics in $\mathbb{P}^{5}$ (see \cite[\S 6.1.4]{Ba94}):
\begin{itemize}
\item[--] according to \cite[p. 21]{Ba94}, the number of moduli of these threefolds is $3$;
\item[--] it is also found by Sano (see \cite[Theorem 1.1 No.12]{Sa95});
\item[--] it is \textit{rational} (see \cite[Corollary 7.2]{Ch97});
\end{itemize}

\item\label{bayle10} the Enriques-Fano threefold $W_{BS}^{10}$ of genus $p=10$ given by the quotient 
of $\mathbb{P}^1 \times S_{6}$, where $S_{6}$ is a smooth Del Pezzo surface of degree $6$ in $\mathbb{P}^{6}$ (see \cite[\S 6.5.1]{Ba94}):
\begin{itemize}
\item[--] this threefold has no moduli (see \cite[p. 56]{Ba94});
\item[--] it is also found by Sano (see \cite[Theorem 1.1 No.13]{Sa95});
\item[--] it is \textit{rational} (see \cite[Corollary 7.2]{Ch97});
\end{itemize}

\item\label{bayle13fano} the Enriques-Fano threefold $W_{BS}^{13}$ of genus $p=13$ given by the quotient 
of $\mathbb{P}^{1} \times \mathbb{P}^{1} \times \mathbb{P}^{1}$ (see \cite[\S 6.3.2]{Ba94}):
\begin{itemize}
\item[--] this threefold has no moduli (see \cite[p. 37]{Ba94});
\item[--] it is also found by Sano (see \cite[Theorem 1.1 No.14]{Sa95});
\item[--] it is \textit{rational} (see \cite[Corollary 7.2]{Ch97});
\end{itemize}
\end{enumerate}

\begin{remark}
Sano found another threefold (see \cite[Theorem 1.1 No.6]{Sa95}) but Bayle excluded it, by providing a more accurate analysis than Sano's (see \cite[\S 6.2.5]{Ba94}).
\end{remark}

\begin{remark}
If $(W,\mathcal{L})$ is one among $W_{BS}^6$, $W_{BS}^7$, $W_{BS}^8$, $W_{BS}^9$, $W_{BS}^{10}$ and $W_{BS}^{13}$, then the 
elements of $\mathcal{L}$ are very ample (see \cite[Theorem A]{Ba94}).
\end{remark}

More generally, if an Enriques-Fano threefold has terminal singularities, then it admits a $\mathbb{Q}$\textit{-smoothing}, i.e., it appears as central fibre of a small deformation over the $1$-parameter unit disk such that a general fibre only has cyclic quotient terminal singularities (see \cite[Main Theorem 2]{Mi99}). Hence every Enriques-Fano threefold with only terminal singularities is a limit of some found by Bayle and Sano. 

\begin{remark}\label{rem:linkBayleFano}
Since the rational F-EF 3-folds have eight quadruple points whose tangent cone is a cone over a Veronese surface (see \cite[p. 44]{Fa38}), then they only have terminal singularities (see \cite[Example 1.3]{Re87}) and therefore they are limits of BS-EF 3-folds (see \cite[Main Theorem 2]{Mi99}). 
\end{remark}

Thus, to complete the classification, one has to consider the case of non-terminal canonical singularities. However, only a few examples of Enriques-Fano threefolds with non-terminal canonical singularities are known: one of genus $p=9$ found by Knutsen, Lopez and Mu\~{n}oz (see \cite[Proposition 13.1]{KLM11}) and two others found by Prokhorov having genus $p=17$ and $p=13$ (see \cite[\S 3]{Pro07}). 
In particular, the one introduced in \cite{KLM11}, which we will call \textit{KLM-EF 3-fold}, is the following:
\begin{enumerate}[(I)]
\setcounter{enumi}{14}
\item\label{KLM} the Enriques-Fano threefold $W_{KLM}^9 \subset \mathbb{P}^9$ of genus $p=9$ given by the image of the F-EF 3-fold $W_F^{13}\subset \mathbb{P}^{13}$ via the rational map $\rho_{\left\langle E_3 \right\rangle} : \mathbb{P}^{13} \dashrightarrow \mathbb{P}^{9}$, which is the projection of $\mathbb{P}^{13}$ from the three-dimensional linear subspace $\mathbb{P}^{3}\cong \left\langle E_3 \right\rangle$ spanned by a certain smooth irreducible elliptic quartic curve $E_3\subset W_F^{13}$;
\end{enumerate}
the ones constructed in \cite{Pro07}, which we will call \textit{P-EF 3-folds}, are the following:
\begin{enumerate}[(I)] 
\setcounter{enumi}{15}
\item\label{pro13} an Enriques-Fano threefold $W_{P}^{13}$ of genus $p=13$ given by the quotient of a cone over a smooth Del Pezzo surface of degree $6$, under an involution fixing five points: it was mentioned very briefly in \cite[Remark 3.3]{Pro07}, so we refer to \S~\ref{subsec:pro13} for more details; this threefold is (at least) \textit{unirational};
\item\label{pro17} an Enriques-Fano threefold $W_{P}^{17}$ of genus $p=17$ given by the quotient of a cone over the octic Del Pezzo surface obtained by the anticanonical embedding of $\mathbb{P}^{1}\times \mathbb{P}^{1}$ in $\mathbb{P}^8$, under an involution fixing five points (see \cite[Proposition 3.2]{Pro07}); this threefold is (at least) \textit{unirational}.
\end{enumerate}

\section{2-divisibility of the curve sections of the F-EF 3-fold of genus 13}\label{subsec:Fano13}

In this section we prove the 2-divisibility of the curve sections of the F-EF 3-fold of genus $13$, which we will use in the proof of Theorem~\ref{thm:SID} to determine their SID. In order to do this, we introduce the geometric construction of the F-EF 3-fold of genus $13$.

Let us take a tetrahedron $T\subset \mathbb{P}^3$. Let $v_0$, $v_1$, $v_2$, $v_3$ be the vertices of $T$ and let $f_i$ be the face opposite to the vertex $v_i$, for $0\le i\le 3$. Furthermore, let us denote the edges of $T$ by $l_{ij}:=f_i\cap f_j$, for $0\le i<j\le 3$. Let us consider the linear system $\mathcal{S}$ of the sextic surfaces of $\mathbb{P}^{3}$ having double points along the six edges of $T$.
Up to a change of coordinates, we may assume to have the tetrahedron $T=\{s_0s_1s_2s_3=0\}$ in $\mathbb{P}^3_{\left[ s_0:s_1:s_2:s_3 \right]}$ with faces $f_i = \{s_i=0\}$, for $0\le i\le 3$. The linear system $\mathcal{S}$ is therefore defined by the zero locus of the following homogeneous polynomial
$$\lambda_0s_1^2s_2^2s_3^2+ \lambda_1s_0^2s_2^2s_3^2+\lambda_2s_0^2s_1^2s_3^2+\lambda_3s_0^2s_1^2s_2^2+ s_0s_1s_2s_3Q(s_0,s_1,s_2,s_3),$$
where $\lambda_0,\lambda_1,\lambda_2,\lambda_3\in \mathbb{C}$ and $Q(s_0,s_1,s_2,s_3)=\sum_{i\le j} q_{ij}s_is_j$ is a quadratic form (see \cite[p.635]{GH}).
Since $\dim H^0(\mathbb{P}^3, \mathcal{O}_{\mathbb{P}^3}(2)) = \binom{3+2}{2}$, then $\dim \mathcal{S} = 13$. 
\begin{remark}\label{rem:variabeleTC13}
Let $\Sigma$ be a general element of $\mathcal{S}$ and let us fix four distinct indices $i,j,k,h\in\{0,1,2,3\}$. By looking locally at the equation of $\mathcal{S}$, one obtains that: 
\begin{itemize}
\item[(i)] $\Sigma$ has triple points at the vertices of $T$ and the tangent cone to $\Sigma$ at $v_i$ is $f_j\cup f_k \cup f_h$;
\item[(ii)] if $p\in l_{ij}$, with $p\ne v_k$ and $p\ne v_h$, then the tangent cone to $\Sigma$ at $p$ is the union of two variable planes containing $l_{ij}$, depending on the choice of the point $p$ and of the surface $\Sigma$, and coinciding for finitely many points $p$.
\end{itemize}
\end{remark}
The rational map $\nu_{\mathcal{S}} : \mathbb{P}^{3} \dashrightarrow \mathbb{P}^{13}$ defined by $\mathcal{S}$ is birational onto the image, which is the F-EF 3-fold $W_F^{13}$ (see \cite[\S 8]{Fa38}).

\begin{theorem}\label{thm:divisibilitySigma}
Let $S$ be a general hyperplane section of the 
F-EF 3-fold $W_F^{13}\subset \mathbb{P}^{13}$ of genus $13$. Then a general curve section of $W_F^{13}\subset \mathbb{P}^{13}$ on $S$ is $2$-divisible in $\operatorname{Pic}(S)$. 
\end{theorem}
\begin{proof}
The idea of the proof is to blow-up $\mathbb{P}^3$ along the base locus of $\mathcal{S}$, until one obtains a smooth rational threefold $Y$ and a base point free linear system $\widetilde{\mathcal{S}}$ on $Y$, which defines a birational morphism $\nu_{\widetilde{\mathcal{S}}}: Y \to W_F^{13}\subset \mathbb{P}^{13}$. Let $S$ be a general hyperplane section of $W_F^{13}\subset \mathbb{P}^{13}$ and let $\widetilde{\Sigma}$ be the element of $\widetilde{\mathcal{S}}$ such that $\nu_{\widetilde{\mathcal{S}}}(\widetilde{\Sigma}) = S$. We will show that a general curve section of $W_F^{13}\subset \mathbb{P}^{13}$ on $S$ is $2$-divisible, since it corresponds to a divisor on $\widetilde{\Sigma}$ which is $2$-divisible. 

First we blow-up $\mathbb{P}^3$ at the vertices of $T$, obtaining a smooth threefold $Y'$ and a birational morphism $bl' : Y'\to \mathbb{P}^{3}$ with exceptional divisors
$E_i := (bl')^{-1}(v_i)$, for $0\le i \le 3$. 
Let $\mathcal{S}'$ be the strict transform of $\mathcal{S}$ and let us denote by $\mathfrak{h}$ the pullback on $Y'$ of the hyperplane class on $\mathbb{P}^{3}$. Then an 
element of $\mathcal{S}'$ is linearly equivalent to $6\mathfrak{h}-3\sum_{i=0}^{3}E_i$.
Let $\widetilde{f}_i$ be the strict transform of the face $f_i$, for $0\le i \le 3$. We denote by $\gamma_{ij}:=E_i\cap \widetilde{f}_j$ the line cut out by $\widetilde{f}_j$ on $E_i$, for $0\le i<j \le 3$. 
If $\Sigma'$ is the strict transform of a general $\Sigma\in\mathcal{S}$, then $\Sigma'\cap E_i = \bigcup_{\substack{j=0 \\ j\ne i}}^3 \gamma_{ij}$, for all $0\le i \le 3$, and $\Sigma'$ is smooth at a general point of $\gamma_{ij}$
(see Remark~\ref{rem:variabeleTC13}). 
The base locus of $\mathcal{S}'$ is therefore given by the union of the strict transforms $\widetilde{l}_{ij}$ of the six edges of $T$ (along which a general $\Sigma'\in \mathcal{S}'$ has double points) and the $12$ lines $\gamma_{ij}$
(see Remark~\ref{rem:variabeleTC13}).
Let us now blow-up $Y'$ along the strict transforms of the edges of $T$: we obtain a smooth threefold $Y''$ and a birational morphism $bl'' : Y'' \to Y'$ with exceptional divisors 
$F_{ij}:=(bl'')^{-1}(\widetilde{l}_{ij}),$ for $0\le i<j \le 3$.
Let $\mathcal{S}''$ be the strict transform of $\mathcal{S}'$: an 
element of $\mathcal{S}''$ is linearly equivalent to $6\mathfrak{h}-3\sum_{i=0}^3 \widetilde{E}_i-2\sum_{0\le i < j \le 3} F_{ij}$, where $\widetilde{E}_i$ is the strict transform of $E_i$, for $0\le i \le 3$, and where $\mathfrak{h}$ denotes the pullback $bl''^* \mathfrak{h}$, by abuse of notation. 
Furthermore, the base locus of $\mathcal{S}''$ is given by the disjoint union of the strict transforms $\widetilde{\gamma}_{ij}$ of the 12 lines $\gamma_{ij}$, for $i,j\in\{0,1,2,3\}$ and $i\ne j$
(see Remark~\ref{rem:variabeleTC13}).
Finally let us consider $bl''' : Y \to Y''$ the blow-up of $Y''$ along the twelve curves $\widetilde{\gamma}_{ij}$, for $i,j\in\{0,1,2,3\}$ and $i\ne j$, with exceptional divisors $\Gamma_{ij}:=bl'''^{-1}(\widetilde{\gamma}_{ij})$. 
We denote by $\mathcal{E}_i$ the strict transform of $\widetilde{E}_i$, by $\mathcal{F}_{ij}$ the strict transform of $F_{ij}$ and by $\mathcal{H}$ the pullback of $\mathfrak{h}$, for $0\le i<j \le 3$. 
Let $\widetilde{\Sigma}$ be the strict transform on $Y$ of an 
element of $\mathcal{S}''$: then
$$\widetilde{\Sigma}\sim 6\mathcal{H}-\sum_{i=0}^3 3\mathcal{E}_k-\sum_{0\le i< j\le 3}2\mathcal{F}_{ij}-\sum_{ \substack{ i,j=0 \\ i\ne j } }^3 4\Gamma_{ij}.$$
Let us take the linear system $\widetilde{\mathcal{S}}:=|\mathcal{O}_{Y}(\widetilde{\Sigma})|$ on $Y$. It is base point free and it defines a morphism $\nu_{\widetilde{\mathcal{S}}}: Y \to \mathbb{P}^{13}$ birational onto the image, which is the F-EF 3-fold $W_F^{13}$. In particular we have the following diagram:
$$\begin{tikzcd}
Y \arrow[d, "bl'''"] \arrow[drrr, "\nu_{\widetilde{\mathcal{S}}}"] & & & \\
Y''  \arrow{r}{bl''} & Y' \arrow{r}{bl'} & \mathbb{P}^3 \arrow[dashrightarrow]{r}{\nu_{\mathcal{S}}} & W_F^{13} \subset \mathbb{P}^{13}.
\end{tikzcd}$$
Let $\widetilde{\Sigma}$ be the element of $\widetilde{\mathcal{S}}$ such that $\nu_{\widetilde{\mathcal{S}}}(\widetilde{\Sigma}) = S$, where $S$ is a general hyperplane section of $W_F^{13}\subset \mathbb{P}^{13}$. By construction we have $\widetilde{\Sigma}\cdot \mathcal{E}_i = 0$ for all $0\le i \le 3$. A general curve section of $W_F^{13}$ on $S$ corresponds to the following divisor on $\widetilde{\Sigma}$: $(6\mathcal{H}-\sum_{0\le i< j\le 3}2\mathcal{F}_{ij}-\sum_{ \substack{ i,j=0 \\ i\ne j } }^2 4\Gamma_{ij})|_{\widetilde{\Sigma}}$, which is $2$-divisible. Thus, we obtain the assertion. 
\end{proof}

\section{2-divisibility of the curve sections of the F-EF 3-fold of genus 9}\label{subsec:Fano9}

In this section we prove the 2-divisibility of the curve sections of the F-EF 3-fold of genus $9$, which we will use in the proof of Theorem~\ref{thm:SID} to determine their SID. In order to do this, we introduce the geometric construction of the F-EF 3-fold of genus $9$.

Let us take the trihedron $T\subset \mathbb{P}^3$ with vertex $v$, faces $f_i$ and edges $l_{ij} := f_i \cap f_j$, and the trihedron $T'\subset \mathbb{P}^3$ with vertex $v'$, faces $f_i'$ and edges $l_{ij}' := f_i' \cap f_j'$, for $1\le i <j \le 3$. 
Let us consider the linear system $\mathcal{K}$ of the septic surfaces of $\mathbb{P}^{3}$ having double points along the six edges of the two trihedra $T$ and $T'$.

\begin{remark}\label{rem:rij contained in K}
A septic surface $K\in \mathcal{K}$ contains the nine lines $r_{ij}:=f_i\cap f_j'$, for $i,j\in\{1,2,3\}$. Assume the contrary: then, by Bezout's Theorem, $K\cap r_{ij}$ is given by $7$ points. Furthermore, each line $r_{ij}$ intersects two edges of $T$ contained in $f_i$ and two edges of $T'$ contained in $f_j'$. Hence $r_{ij}$ is a line through four double points of $K$.
We obtain that $K\cap r_{ij}$ contains at least $8$ points, counted with multiplicity, which is a contradiction. Thus, it must be $r_{ij} \subset K$. 
\end{remark}

\begin{proposition}\label{prop:dimK=9}
The linear system $\mathcal{K}$ is defined by the zero locus of the following homogeneous polynomial of degree seven
$$F(s_0,s_1,s_2,s_3) = f_1f_2f_3f_1'f_2'f_3'(\lambda_0s_0+\lambda_1s_1+\lambda_2s_2+\lambda_3s_3)+$$
$$+f_1'f_2'f_3'(\lambda_4 f_3^2f_2^2+\lambda_5 f_1^2f_3^2+\lambda_6 f_1^2f_2^2)+f_1f_2f_3(\lambda_7 {f_3'}^2{f_2'}^2+\lambda_8 {f_1'}^2{f_3'}^2+\lambda_9 {f_1'}^2{f_2'}^2),$$
where $\lambda_0,\dots ,\lambda_9\in \mathbb{C}$, and where $f_i$ and $f_i'$ denote, by abuse of notation, the linear homogeneous polynomials defining, respectively, the faces $f_i$ and $f_i'$, for $1\le i \le 3$. The linear system $\mathcal{K}$ therefore has $\dim \mathcal{K} = 9$.
\end{proposition} 
\begin{proof}
Let $F\in \mathbb{C}\left[ s_0,s_1,s_2,s_3 \right]$ be the homogeneous polynomial of degree $7$ defining a general element $K$ of $\mathcal{K}$. 
The surface $K$ intersects each face $f_i$ of $T$ along the septic curve given by the two double edges contained in that face plus the three lines $r_{ij}$, for $1\le j \le 3$. The same happens with the faces of $T'$. 
This implies that it must be $K\cap f_i=\{f_1'f_2'f_3'f_k^2f_h^2=0,\, f_i = 0\}=2l_{ik}+2l_{ih}+\sum_{j=1}^3r_{ij}$ and $K\cap f_i'=\{f_1f_2f_3{f_k'}^2{f_h'}^2=0,\, f_i' = 0\}= 2l_{ik}'+2l_{ih}'+\sum_{j=1}^3r_{ji}$, for distinct indices $i,k,h\in\{1,2,3\}$. 
Then it must be
$$F(s_0,s_1,s_2,s_3) = f_1g_6(s_0,s_1,s_2,s_3)+\lambda_4 f_1'f_2'f_3'f_3^2f_2^2,$$
where $\lambda_4\in \mathbb{C}$ and $g_6$ is a homogeneous polynomial of degree $6$ such that
$$g_6(s_0,s_1,s_2,s_3) = f_2g_5(s_0,s_1,s_2,s_3)+\lambda_5 f_1'f_2'f_3'f_1f_3^2,$$
where $\lambda_5\in \mathbb{C}$ and $g_5$ is a homogeneous polynomial of degree $5$ such that
$$g_5(s_0,s_1,s_2,s_3) = f_3g_4(s_0,s_1,s_2,s_3)+\lambda_6 f_1'f_2'f_3'f_1f_2,$$
where $\lambda_6\in \mathbb{C}$ and $g_4$ is a homogeneous polynomial of degree $4$ such that
$$g_4(s_0,s_1,s_2,s_3) = f_1'g_3(s_0,s_1,s_2,s_3)+\lambda_7 f_2'^2f_3'^2,$$
where $\lambda_7\in \mathbb{C}$ and $g_3$ is a homogeneous polynomial of degree $3$ such that
$$g_3(s_0,s_1,s_2,s_3) = f_2'g_3(s_0,s_1,s_2,s_3)+\lambda_8 f_1'f_3'^2,$$
where $\lambda_8\in \mathbb{C}$ and $g_2$ is a homogeneous polynomial of degree $2$ such that
$$g_2(s_0,s_1,s_2,s_3) = f_3'(\lambda_0s_0+\lambda_1s_1+\lambda_2s_2+\lambda_3s_3)+\lambda_9 f_1'f_2',$$
where $\lambda_0, \lambda_1, \lambda_2, \lambda_3, \lambda_9\in \mathbb{C}$.
So $F$ has the expression of the statement. Since $\{K\in \mathcal{K}| K\supset f_1\} = \{F=0 | \lambda_4 = 0\}$, then $\operatorname{codim}\left(\{K\in \mathcal{K}| K\supset f_1\},\mathcal{K}\right)=1$. 
Let us see that containing the six faces $f_1$, $f_2$, $f_3$, $f_1'$, $f_2'$, $f_3'$ imposes independent conditions: there exists a septic surface in $\mathcal{K}$ containing $f_1$ but not $f_2$, that is
$\{F=0 | \lambda_4 = 0, \lambda_5 \ne 0\}$;
there exists a septic surface in $\mathcal{K}$ containing $f_1$ and $f_2$ but not $f_3$, that is
$\{F=0 | \lambda_4 = \lambda_5 = 0, \lambda_6\ne 0\}$;
there exists a septic surface in $\mathcal{K}$ containing $f_1$, $f_2$ and $f_3$ but not $f_1'$, that is
$\{F=0 | \lambda_4 = \lambda_5 = \lambda_6 = 0, \lambda_7 \ne 0\}$;
there exists a septic surface in $\mathcal{K}$ containing $f_1$, $f_2$, $f_3$ and $f_1'$ but not $f_2'$, that is
$\{F=0 | \lambda_4 = \lambda_5 = \lambda_6 = \lambda_7 = 0, \lambda_8 \ne 0\}$;
there exists a septic surface in $\mathcal{K}$ containing $f_1$, $f_2$, $f_3$, $f_1'$, and $f_2'$ but not $f_3'$, that is
$$\{F=0 | \lambda_4 = \lambda_5 = \lambda_6 = \lambda_7 = \lambda_8 = 0, \lambda_9 \ne 0\}.$$
Thus, we obtain $\operatorname{codim}(\{K\in \mathcal{K}| K\supset T\cup T'\},\mathcal{K})=6$.
Furthermore, each element of $\{K\in \mathcal{K}| K\supset T\cup T'\}$ is of the form $T\cup T'\cup \pi$, where $\pi$ is a general plane of $\mathbb{P}^3$. Thus, we have $\dim \{K\in \mathcal{K}| K\supset T\cup T'\}= \dim |\mathcal{O}_{\mathbb{P}^3}(1)|=3$ and finally $\dim \mathcal{K}=3+6=9$.
\end{proof}

Let us consider the points mentioned in Remark~\ref{rem:rij contained in K}: they are
$q_{ijk}:=l_{ij}\cap r_{ik} = l_{ij}\cap r_{jk}$ and $q_{ijk}':=l_{ij}'\cap r_{ki}=l_{ij}'\cap r_{kj}$, for $i,j,k\in\{1,2,3\}$ with $i<j$. 
These points also represent the intersection points between the faces of a trihedron and the edges of the other trihedron: indeed, we have that $q_{ijk} = l_{ij}\cap f_k'$ and $q_{ijk}' =l_{ij}'\cap f_k$.

\begin{remark}\label{rem:variabeleTC9}
Let $K$ be a general element of $\mathcal{K}$. By looking locally at the equation of $\mathcal{K}$ (see Proposition~\ref{prop:dimK=9}), one can find that:
\begin{itemize}
\item[(i)] $K$ has triple points at the vertices of $T$ and $T'$ and the tangent cones to $K$ at $v$ and $v'$ are, respectively, $\bigcup_{i=1}^3f_i$ and $\bigcup_{i=1}^3f_i'$;
\item[(ii)] the tangent cone to $K$ at $q_{ijk}$ is $f_i\cup f_j$ and the tangent cone to $K$ at $q_{ijk}'$ is $f_i'\cup f_j'$, for $i,j,k\in\{1,2,3\}$ with $i<j$;
\item[(iii)] if $p\in l_{ij}$, with $p\ne v$ and $p\ne q_{ijk}$, then the tangent cone to $K$ at $p$ is the union of two variable planes containing $l_{ij}$, depending on the choice of the point $p$ and of the surface $K$, and coinciding for finitely many points $p$. Similarly if $p\in l_{ij}'$, with $p\ne v'$ and $p\ne q_{ijk}'$, then the tangent cone to $K$ at $p$ is the union of two elements of $|\mathcal{I}_{l_{ij}'|\mathbb{P}^3}(1)|$ that depend on the choice of $p$ and of $K$ and that can also coincide for finitely many points $p$;
\item[(iv)] $K$ is smooth along $r_{ik}$, except at the points contained in the edges of the two trihedra.
\end{itemize}
\end{remark}

The rational map $\nu_{\mathcal{K}} : \mathbb{P}^{3} \dashrightarrow \mathbb{P}^{9}$ defined by $\mathcal{K}$ is birational onto the image, which is the F-EF 3-fold $W_F^{9}$ of genus $9$ (see \cite[\S 7]{Fa38}).
\begin{theorem}\label{thm:divisibilityK}
Let $S$ be a general hyperplane section of the F-EF 3-fold 
$W_F^{9}\subset \mathbb{P}^{9}$ 
of genus of $9$. Then a general curve section of $W_F^{9}\subset \mathbb{P}^9$ on $S$ is $2$-divisible in $\operatorname{Pic}(S)$. 
\end{theorem}
\begin{proof}
The idea of the proof is similar to the one of the proof of Theorem~\ref{thm:divisibilitySigma}. 
First we blow-up $\mathbb{P}^3$ at the vertices of the trihedra and at the $18$ points $q_{ijk}$ and $q_{ijk}'$, for $i,j,k\in\{1,2,3\}$ and $i<j$. We obtain a smooth threefold $Y'$ and a birational morphism $bl' : Y '\to \mathbb{P}^{3}$ with exceptional divisors $E := (bl')^{-1}(v)$, $E' := (bl')^{-1}(v')$, $E_{ijk}:=(bl')^{-1}(q_{ijk})$, $E_{ijk}':=(bl')^{-1}(q_{ijk}')$. 
Let $\mathcal{K}'$ be the strict transform of $\mathcal{K}$ and let us denote by $\mathfrak{h}$ the pullback on $Y'$ of the hyperplane class on $\mathbb{P}^{3}$. Then an 
element of $\mathcal{K}'$ is linearly equivalent to $7\mathfrak{h}-3E-3E'-2\sum_{\substack{i,j,k=1 \\ i<j }}^3(E_{ijk}+E_{ijk}')$. Let $\widetilde{f}_i$ and $\widetilde{f}_i'$ be the strict transforms of the faces $f_i$ and $f_i'$, for $1\le i \le 3$. 
We denote by $\gamma_{i}:=E\cap \widetilde{f}_i$ the line cut out by $\widetilde{f}_i$ on $E$ and by $\gamma_{i}':=E'\cap \widetilde{f}_i'$ the one cut out by $\widetilde{f}_i'$ on $E'$. 
If $K'$ is the strict transform of a general $K\in\mathcal{K}$, then $K'\cap E = \bigcup_{i= 0}^3 \gamma_{i}$ and $K'\cap E' = \bigcup_{i=0}^3 \gamma_{i}'$ and $K'$ is smooth at a general point of $\gamma_{i}$ and of $\gamma_{i}'$ (see Remark~\ref{rem:variabeleTC9}).
We also take the lines $\lambda_{ijk,h}:=E_{ijk}\cap \widetilde{f}_h$ and $\lambda_{ijk,h}':=E_{ijk}'\cap \widetilde{f}_h'$, where $i,j,k\in\{1,2,3\}$ with $i<j$ and $h\in\{i,j\}$. 
We have that $K'\cap E_{ijk} =\bigcup_{h=i,j} \lambda_{ijk,h}$ and $K'\cap E_{ijk}' = \bigcup_{h=i,j}\lambda_{ijk,h}'$ (see Remark~\ref{rem:variabeleTC9}). 
Let us consider the strict transforms $\widetilde{l}_{ij}$, $\widetilde{l}_{ij}'$ and $\widetilde{r}_{ij}$ of the lines $l_{ij}$, $l_{ij}'$ and $r_{ik}$, for $i,j,k\in\{1,2,3\}$ and $i<j$. Then the base locus of $\mathcal{K}'$ is given by the union of the six curves $\widetilde{l}_{ij}$, $\widetilde{l}_{ij}'$ (along which a general $K'\in\mathcal{K}'$ has double points), of the nine curves $\widetilde{r}_{ik}$, of the six lines $\gamma_i$, $\gamma_i'$, and of the $36$ lines $\lambda_{ijk,h}$, $\lambda_{ijk,h}'$ (see Remark~\ref{rem:variabeleTC9}).
Let us blow-up $Y'$ along the strict transforms of the edges of the trihedra and of the nine lines $r_{ij}$. We obtain a smooth threefold $Y''$ and a birational morphism $bl'' : Y'' \to Y'$ with exceptional divisors 
$F_{ij}:=(bl'')^{-1}(\widetilde{l}_{ij})$,
$F_{ij}':=(bl'')^{-1}(\widetilde{l}_{ij}')$,
$R_{ij}:=(bl'')^{-1}(\widetilde{r}_{ij})$.
Let us denote by $\widetilde{E}$, $\widetilde{E}'$, $\widetilde{E}_{ijk}$ and $\widetilde{E}_{ijk}'$ respectively the strict transforms of $E$, $E'$, $E_{ijk}$ and $E_{ijk}'$. 
Let $\mathcal{K}''$ be the strict transform of $\mathcal{K}'$: an 
element of $\mathcal{K}''$ is linearly equivalent to 
$7\mathfrak{h}-3\widetilde{E}-3\widetilde{E}'-2\sum_{\substack{i,j,k=1 \\ i<j }}^3(\widetilde{E}_{ijk}+\widetilde{E}_{ijk}')-2\sum_{1\le i < j \le 3} (F_{ij}+F_{ij}')-\sum_{i,j=1}^3R_{ij},$
where, by abuse of notation, $\mathfrak{h}$ also denotes the pullback $bl''^* \mathfrak{h}$.
By Remark~\ref{rem:variabeleTC9} 
we have that the base locus of $\mathcal{K}''$ is given by the disjoint union of the strict transforms $\widetilde{\gamma}_{i}$, $\widetilde{\gamma}_{i}'$, $\widetilde{\lambda}_{ijk,h}$, $\widetilde{\lambda}_{ijk,h}'$ of the 42 lines defined as above. 
Finally let us consider the blow-up of $Y''$ along these $42$ curves, which is the map $bl''' : Y \to Y''$ with exceptional divisors $\Gamma_{i}:=bl'''^{-1}(\widetilde{\gamma}_{i})$, $\Gamma_{i}':=bl'''^{-1}(\widetilde{\gamma}_{i}')$, $\Lambda_{ijk,h}:=bl'''^{-1}(\widetilde{\lambda}_{ijk,h})$, $\Lambda_{ijk,h}':=bl'''^{-1}(\widetilde{\lambda}_{ijk,h}')$.  
We denote by $\mathcal{E}$, $\mathcal{E}'$, $\mathcal{E}_{ijk}$, $\mathcal{E}_{ijk}'$, respectively, the strict transform of $\widetilde{E}$, $\widetilde{E}'$, $\widetilde{E}_{ijk}$, $\widetilde{E}_{ijk}'$; by $\mathcal{F}_{ij}$ the strict transform of $F_{ij}$; by $\mathcal{R}_{ik}$ the strict transform of $R_{ik}$; by $\mathcal{H}$ the pullback of $\mathfrak{h}$, for $i,j,k\in\{1,2,3\}$ with $i< j$ and $h\in\{i,j\}$. 
Let $\widetilde{K}$ be the strict transform on $Y$ of an 
element of $\mathcal{K}''$: then
$$\widetilde{K} \sim 7\mathcal{H}-3\mathcal{E}-3\mathcal{E}'-2\sum_{\substack{i,j,k=1 \\ i<j }}^3(\mathcal{E}_{ijk}+\mathcal{E}_{ijk}')-2\sum_{1\le i < j \le 3}(\mathcal{F}_{ij}+\mathcal{F}_{ij}')-\sum_{i,j=1}^{3}\mathcal{R}_{ij}+$$
$$-4\sum_{i=1}^{3}(\Gamma_i+\Gamma_i')-3\sum_{ \substack{i,j,k=1 \\ i<j,\, h=i,j }}^3 (\Lambda_{ijk,h}+\Lambda_{ijk,h}').$$
Let us take the linear system $\widetilde{\mathcal{K}}:=|\mathcal{O}_{Y}(\widetilde{K})|$ on $Y$. It is base point free and it defines a morphism $\nu_{\widetilde{\mathcal{K}}}: Y \to \mathbb{P}^{9}$ birational onto the image, which is the F-EF 3-fold $W_F^{9}$. 
In particular we have the following diagram:
$$\begin{tikzcd}
Y \arrow[d, "bl'''"] \arrow[drrr, "\nu_{\widetilde{\mathcal{K}}}"] & & & \\
Y''  \arrow{r}{bl''} & Y' \arrow{r}{bl'} & \mathbb{P}^3 \arrow[dashrightarrow]{r}{\nu_{\mathcal{K}}} & W_F^{9} \subset \mathbb{P}^{9}.
\end{tikzcd}$$
We observe that there is only one cubic surface in $\mathbb{P}^3$ which is singular along the edges of the trihedron $T'$, that is $T'$ itself. Let us consider its strict transform on $Y$, which is
$$\widetilde{T}'\sim 3\mathcal{H}-3\mathcal{E}'-\sum_{\substack{i,j,k=1 \\ i<j}}^{3}(\mathcal{E}_{ijk}+2\mathcal{E}_{ijk}')-\sum_{i=1}^{3}2\mathcal{F}_{ij}'-\sum_{i,j=1}^{3}\mathcal{R}_{ij}+$$
$$-\sum_{i=1}^{3}4\Gamma_i'-\sum_{ \substack{ i,j,k\in \{1,2,3\} \\ i< j,\,\, h=i,j } }(\Lambda_{ijk,h}+3\Lambda_{ijk,h}').$$
Let $\widetilde{K}$ be the divisor of $\mathcal{\widetilde{K}}$ such that $\nu_{\widetilde{\mathcal{K}}}(\widetilde{K})=S$, where $S$ is a general hyperplane section of $W_F^{9}\subset \mathbb{P}^{9}$. By construction we have
$\widetilde{K}\cdot \mathcal{E} = \widetilde{K}\cdot \mathcal{E} = \widetilde{K}\cdot \mathcal{E}_{ijk} = \widetilde{K}\cdot \mathcal{E}_{ijk}' = 0$, for all $i,j,k\in\{1,2,3\}$ with $i<j$.
Thus, we have that
$$0\sim \widetilde{T}'|_{\widetilde{K}}\sim (3\mathcal{H}-\sum_{i=1}^{3}2\mathcal{F}_{ij}'-\sum_{i,j=1}^{3}\mathcal{R}_{ij}-\sum_{i=1}^{3}4\Gamma_i'-\sum_{ \substack{ i,j,k\in \{1,2,3\} \\ i< j,\,\, h=i,j } }(\Lambda_{ijk,h}+3\Lambda_{ijk,h}'))|_{\widetilde{K}}.$$
Then a general curve section of $W_F^9\subset \mathbb{P}^9$ on $S$ corresponds to the following divisor on $\widetilde{K}$:
$$\Big(7\mathcal{H}-\sum_{1\le i < j \le 3}2(\mathcal{F}_{ij}+\mathcal{F}_{ij}')-\sum_{i,j=1}^{3}\mathcal{R}_{ij}-\sum_{i=1}^{3}4(\Gamma_i+\Gamma_i')-\sum_{ \substack{i,j,k=1 \\ i<j,\\ h=i,j }}^3 3(\Lambda_{ijk,h}+\Lambda_{ijk,h}')\Big)|_{\widetilde{K}}\sim $$
$$\sim (4\mathcal{H}-\sum_{1\le i < j \le 3}2\mathcal{F}_{ij}-\sum_{i=1}^{3}4\Gamma_i-\sum_{ \substack{i,j,k=1 \\ i<j}}^3 2\Lambda_{ijk,i})|_{\widetilde{K}},$$
which is $2$-divisible. Thus, we obtain the assertion. 
\end{proof}

\section{Numerical equivalence class of the curve sections of the P-EF 3-fold of genus 13}\label{subsec:pro13}

In this section we prove the numerical 2-divisibility of the curve sections of the P-EF 3-fold of genus $13$, which we will use in the proof of Theorem~\ref{thm:SID} to determine their possible SID. In order to do this, we describe, in more details, the P-EF 3-fold $W_P^{13}$, since it was mentioned very briefly in \cite[Remark 3.3]{Pro07}. We use the techniques used by Prokhorov in \cite[\S 3]{Pro07} for the construction of the P-EF 3-fold $W_P^{17}$.

Let us consider the linear system of the plane cubic curves passing through three fixed points $a_1$, $a_2$, $a_3$ in general position. Up to a change of coordinates, we may assume $a_1=[1:0:0]$, $a_2=[0:1:0]$ and $a_3=[0:0:1]$ in $\mathbb{P}^2_{[u_0:u_1:u_2]}$. The afore-mentioned linear system so defines the rational map $\lambda : \mathbb{P}^2 \dashrightarrow \mathbb{P}^{6}$ given by
$$\left[u_0 : u_1 : u_2\right] \mapsto \left[u_1^2 u_2 : u_1 u_2^2 : u_0^2 u_2 : u_0 u_2^2 : u_0^{2} u_1 : u_0 u_1^2 : u_0 u_1 u_2\right],$$
whose image is a smooth sextic Del Pezzo surface $S_6\subset \mathbb{P}^{6}$. If $bl : \operatorname{Bl}_{a_1,a_2,a_3}\mathbb{P}^2 \to \mathbb{P}^2$ denotes the blow-up of the plane at the three fixed points, then we have that $S_6$ is isomorphic to $\operatorname{Bl}_{a_1,a_2,a_3}\mathbb{P}^2$ and that it is anticanonically embedded in $\mathbb{P}^{6}$. Let $\ell$ be the pullback of the line class on $\mathbb{P}^{2}$ and let $e_i := bl^{-1}(a_i)$ be the exceptional divisors, for $1\le i \le 3$; then we have the following commutative diagram
$$\begin{tikzcd}
\operatorname{Bl}_{a_1,a_2,a_3}\mathbb{P}^{2} \arrow[d, "bl"] \arrow[dr, "\cong"', "\widetilde{\lambda}_{|3\ell - e_1-e_2-e_3|} = \widetilde{\lambda}_{|-K_{S_6}|}"] & \\
\mathbb{P}^{2} \arrow[r, dashrightarrow, "\lambda"']  & S_6 \subset \mathbb{P}^{6}.
\end{tikzcd}$$
Let us consider $\mathbb{P}^{6}_{[x_0:x_1:x_2:x_3:x_4:x_5:x_6]}$ as the hyperplane $\{y_0=0\} \subset \mathbb{P}^7_{[x_0:\dots:x_6:y]}$ and let us take the cone $V$ over $S_6$ with vertex $v:=[0:0:0:0:0:0:0:1]$.

\begin{remark}\label{rem: ideal Vp13}
Since the ideal of $S_6$ is generated by the following polynomials
$$x_3x_5-x_6^2,\quad x_2x_5-x_4x_6,\quad x_1x_5-x_0x_6,\quad x_3x_4-x_2x_6,\quad x_1x_4-x_6^2,$$
$$x_0x_4-x_5x_6,\quad x_0x_3-x_1x_6,\quad x_1x_2-x_3x_6,\quad x_0x_2-x_6^2,$$
in $\mathbb{C}[x_0,x_1,x_2,x_3,x_4,x_5,x_6]$, then the ideal of $V$ is generated by the same polynomials as polynomials in $\mathbb{C}[x_0,x_1,x_2,x_3,x_4,x_5,x_6,y]$.
\end{remark}

\begin{lemma}\label{lem:V withcanonicalvertex 13}
The variety $V$ is a Gorenstein Fano threefold with canonical singularities. Moreover, $-K_V = 2M$ where $M$ is the class of the hyperplane sections.
\end{lemma}
\begin{proof}
Since $S_6\subset \mathbb{P}^{6}$ is projectively normal (see \cite[Theorem 8.3.4]{Dolg12}), then $V$ is normal. 
Let $\sigma : \operatorname{Bl}_{v}V \to V$ be the blow-up of $v$ with exceptional divisor $E=\sigma^{-1}(v)$. Then $\operatorname{Bl}_{v}V$ is a $\mathbb{P}^{1}$-bundle over $S_6$ and $\sigma$ contracts its negative section $E$ to $v$. In particular we have $\operatorname{Bl}_{v}V = \mathbb{P}\left(\mathcal{O}_{S_6}\oplus \mathcal{O}_{S_6}(-K_{S_6})\right)$ (see \cite[V, Ex. 2.11.4]{Hart}). Since the map $\sigma : \operatorname{Bl}_{v}V \to V \subset \mathbb{P}^{7}$ is given by the tautological linear system $|\mathcal{O}_{\operatorname{Bl}_{v}V}(1)|$, then $\mathcal{O}_{\operatorname{Bl}_{v}V}(1)\sim \sigma^* M$.
A priori we have that $K_{\operatorname{Bl}_{v}V}=\sigma^* K_V+aE$ for $a\in \mathbb{Q}$.  
Since $K_{\operatorname{Bl}_{v}V}\sim \mathcal{O}_{\operatorname{Bl}_{v}V}(-2)$ (see \cite[p. 349 (d)]{Re87}) and $K_{\operatorname{Bl}_{v}V}\cdot E =-2 (\sigma^* M) \cdot E =0$, then $a=0$ and $K_V$ is a Cartier divisor. 
Thus, $V$ has a canonical singularity at the vertex $v$. Finally, since $\sigma^* (-2M) \sim \mathcal{O}_{\operatorname{Bl}_{v}V}(-2) \sim K_{\operatorname{Bl}_{v}V} = \sigma^* K_V$, we have that $K_V = -2M$.
\end{proof}

The quadratic transformation $q_{a_1,a_2,a_3} : \mathbb{P}^2 \dashrightarrow \mathbb{P}^2$, given by the linear system of the conics passing through $a_1$, $a_2$ and $a_3$, defines an involution of the sextic Del Pezzo $S_6\subset \mathbb{P}^6$. Indeed, we have
\begin{footnotesize}
$$\begin{tikzcd}
\left[u_0 : u_1 : u_2\right] \arrow[r, mapsto, "q_{a_1,a_2,a_3}"] \arrow[d, mapsto, "\lambda"] & \left[\frac{1}{u_0} : \frac{1}{u_1} : \frac{1}{u_2}\right] \arrow[dd, mapsto, "\lambda"]\\
\hspace{-1cm}\left[u_1^2 u_2 : u_1 u_2^2 : u_0^2 u_2 : u_0 u_2^2 : u_0^{2} u_1 : u_0 u_1^2 : u_0 u_1 u_2\right] & \\
& \hspace{-0.6cm} \left[u_0^2 u_2 : u_1 u_0^2 : u_1^2 u_2 : u_0 u_1^2 : u_1 u_2^2 : u_0 u_2^2 : u_0 u_1 u_2\right]
\end{tikzcd}$$
\end{footnotesize}
and then we obtain the involution $t'$ of $S_6\subset \mathbb{P}^{6}$ given by
$$\begin{tikzcd}
\hspace{-1cm}\left[x_0 : x_1 : x_2 : x_3 : x_4 : x_5 : x_6\right] \arrow[r, mapsto, "t' "] & \left[x_2 : x_4 : x_0 : x_5 : x_1 : x_3 : x_6\right].
\end{tikzcd}$$
Let us take the involution of $\mathbb{P}^7$ defined by $t: \mathbb{P}^7 \to \mathbb{P}^7$ such that
$$\begin{tikzcd}
\left[x_0 : x_1 : x_2 : x_3 : x_4 : x_5 : x_6: y\right] \arrow[r, mapsto] & \left[x_2 : x_4 : x_0 : x_5 : x_1 : x_3 : x_6: -y\right].
\end{tikzcd}$$
The locus of $t$-fixed points in $\mathbb{P}^{7}$ consists of two projective subspaces
$$F_1 = \{x_0+x_2 = x_1+x_4 = x_3+x_5 = x_6 = 0\}\cong \mathbb{P}^{3},$$
$$F_{2} = \{x_0-x_2 = x_1-x_4 = x_3-x_5 = y = 0\} \cong \mathbb{P}^{3}.$$
In particular we have that $F_1\cap V = \{v\}$ and $F_{2}\cap V = \{v_1,v_2,v_3,v_4\}$, where
$$v_1 := \left[1:1:1:1:1:1:1:0 \right], \quad v_2 :=\left[1:-1:1:-1:-1:-1:1:0\right],$$
$$v_3 := \left[-1:1:-1:-1:1:-1:1:0\right], \quad v_4 := \left[-1:-1:-1:1:-1:1:1:0\right].$$
Thus, $t$ induces an involution $\tau := t |_{V}$ of $V$ with five fixed points.  

\begin{proposition}\label{prop: definition of pro W13}
The quotient of $V$ by the involution $\tau$ is an Enriques-Fano threefold of genus $p=13$, which we will denote by $W_{P}^{13}$.
\end{proposition}
\begin{proof}
Let $\mathcal{Q}_{V}$ be the linear system that is cut out on $V$ by the linear system $\mathcal{Q}$ of the quadric hypersurfaces of $\mathbb{P}^{7}$ of type
$$q_1(x_0+x_2,x_1+x_4,x_3+x_5,x_6)+q_{2}(x_0-x_2,x_1-x_4,x_3-x_5,y)=0,$$
where $q_1$ and $q_2$ are quadratic homogeneous forms. By construction, we have that $\mathcal{Q}_{V}$ is base point free and each member of $\mathcal{Q}_{V}$ is $\tau$-invariant. In particular a general member $\widetilde{S}\in \mathcal{Q}_{V}$ is smooth and does not contain any of $v,v_1,v_2,v_3,v_4$. Then the action of $\tau$ on $\widetilde{S}$ is fixed point free. Moreover $\widetilde{S}$ is a K3 surface, since $\mathcal{Q}_{V} \subset |2M|=|-K_V|$. Let $\pi : V \to W_{P}^{13} := V/\tau$ be the quotient morphism and let $S:= \pi(\widetilde{S})= \widetilde{S} /\tau$. Then $S$ is a smooth Enriques surface. Since $\widetilde{S}=\pi^* S$, we have $2p-2=S^3=\frac{1}{2}\widetilde{S}^3=\frac{1}{2}(2M)^3=4\cdot \deg V = 24$, whence $p=13$. 
Thus, by setting $\mathcal{L}:=|\mathcal{O}_{W_{P}^{13}}(S)|$, we have that $(W_{P}^{13},\mathcal{L})$ is an Enriques-Fano threefold of genus 13. 
\end{proof}

The linear system $\mathcal{Q}$, introduced in the proof of Proposition~\ref{prop: definition of pro W13}, defines a morphism $\varphi : \mathbb{P}^{7} \to \mathbb{P}^{19}$ given by 
$\left[x_0:x_1:x_2:x_3:x_4:x_5:x_6:y\right] \mapsto \left[Z_0:\dots :Z_{19} \right]$
where
$Z_0 = x_6^2$,
$Z_1 = x_0^2+x_2^2$, 
$Z_2 = x_1^2+x_4^2$, 
$Z_3 = x_3^2+x_5^2$,
$Z_4 = (x_0+x_2)x_6$,
$Z_5 = (x_1+x_4)x_6$,
$Z_6 = (x_3+x_5)x_6$,
$Z_7 = x_0x_1+x_2x_4$,
$Z_8 = x_2x_3+x_0x_5$, 
$Z_9 = x_1x_3+x_4x_5$,
$Z_{10} = (x_0-x_2)y$,
$Z_{11} = (x_1-x_4)y$,
$Z_{12} = (x_3-x_5)y$,
$Z_{13} = y^2$,
$Z_{14} = 2x_0x_2$,
$Z_{15} = 2x_1x_4$,
$Z_{16} = 2x_3x_5$,
$Z_{17} = x_4x_3+x_1x_5$,
$Z_{18} = x_0x_3+x_2x_5$,
$Z_{19} = x_1x_2+x_0x_4$.
Hence we have $\pi = \varphi |_{V} : V \to W_{P}^{13} \subset \mathbb{P}^{19}$. Furthermore, the threefold $W_{P}^{13}$ is contained in a $13$-dimensional projective subspace of $\mathbb{P}^{19}$ given by 
$$H_{13}:= \{Z_{14} = 2Z_0, \,\,
Z_{15} = 2Z_0, \,\,
Z_{16} = 2Z_0, \,\,
Z_{17} = Z_4, \,\,
Z_{18} = Z_5, \,\,
Z_{19} = Z_6\}.$$
%
%
%
(see Remark~\ref{rem: ideal Vp13}). Thus, we obtain $\pi : V \to W_{P}^{13} \subset H_{13} \cong \mathbb{P}^{13}$.

\begin{theorem}\label{thm:SIDpro13}
Let $(W_P^{13},\mathcal{L})$ be the P-EF 3-fold of genus $13$. Let $S\in \mathcal{L}$ be a general hyperplane section of $W_P^{13}$ and let $H$ be a general curve section of $W_P^{13}$ on $S$. Then $\phi(H)=4$ and $H$ is numerically equivalent to $2E_1+2E_2+2E_3$, where $E_1$, $E_2$, $E_3$ are primitive effective isotropic divisors on $S$ such that $E_i \cdot E_j=1$, for $1\le i<j < 3$. 
\end{theorem}

\begin{proof}
Since $W_P^{13}$ has genus $p=13$ and $\phi(H)^2\le H^2 = S^3 = 2p-2$, then we have $1\le\phi(H)\le 4$. Furthermore, we saw that $S\subset \mathbb{P}^{12}$ is $1$-extendable to $W_P^{13}\subset \mathbb{P}^{13}$, so $3\le \phi \le 4$ (see \cite[Theorem 4.6.1]{CoDo89}).
We also recall that $W_P^{13}$ is given by the quotient $\pi : V \to V/ \tau = W_P^{13}$ (see Proposition~\ref{prop: definition of pro W13}).
Since $\widetilde{S}=\pi^* S$ is a K3 surface given by a particular quadric section of $V$ and $S_6\cong \operatorname{Bl}_{a_1,a_2,a_3}\mathbb{P}^2$ is a hyperplane section of $V$, we have a double cover $f : \widetilde{S} \to S_6$ with ramification locus $R_f = f^*(-K_{S_6})=f^*(3\ell -e_1-e_2-e_3)$ and branch locus $B_f = -2K_{S_6} = 6\ell -2e_1-2e_2-2e_3$. Furthermore, by setting $\overline{E}_i := f^{*} (\ell - e_i)$ for $1\le i\le 3$, we have 
$$M|_{\widetilde{S}}\sim f^*(-K_{S_6})=f^*(3\ell -e_1-e_2-e_3)\sim
 f^* (\ell -e_1+\ell -e_2+\ell -e_3)=$$
 $$=f^*(\ell -e_1)+f^*(\ell -e_2)+f^*(\ell -e_3) 
 =\overline{E}_1+\overline{E}_2+\overline{E}_3=:D,$$ 
where $D^2 = 12$, $\overline{E}_i\cdot \overline{E}_j =2$ for $1\le i < j\le 3$ and $\overline{E}_i^2 = 0$. So $\overline{E}_i$ is an elliptic curve for $1\le i\le 3$.
Since
$\pi|_{\widetilde{S}}^* H = (\pi^* S)|_{\widetilde{S}}\sim \widetilde{S}|_{\widetilde{S}}
\sim 2M|_{\widetilde{S}}\sim 2D = 2(\overline{E}_1+\overline{E}_2+\overline{E}_3)$, then $\pi|_{\widetilde{S}}^* H$ is $2$-divisible on the K3 surface $\widetilde{S}$ and $H$ is numerically $2$-divisible on the Enriques surface $S$, i.e. $H$ or $H+K_S$ is $2$-divisible on $S$. Then we only have the following possibile SID: $H\sim 2(E_1+E_2+E_3)$ or $H\sim 2(E_1+E_2+E_3)+K_S$ (see Table~\ref{tab:Ep-phi}). Thus, $\phi(H) = 4$.
\end{proof}

\begin{remark}\label{rem:barEi=tildeEipro13}
Let us use the notation of proof of Theorem~\ref{thm:SIDpro13} and let us consider the elliptic curves $\widetilde{E}_i := \pi^*(E_i)$, for $1\le i\le 3$.
In both cases $H\sim 2(E_1+E_2+E_3)$ and $H\sim 2(E_1+E_2+E_3)+K_S$, we have $\widetilde{H}:=\pi|_{\widetilde{S}}^* H = 2(\widetilde{E}_1+\widetilde{E}_2+\widetilde{E}_3)$, where $\widetilde{E}_i\cdot \widetilde{E}_j= 2$ for $1\le i< j\le 3$. Furthermore, we have that $\overline{E}_i = \widetilde{E}_i$, for $1\le i\le 3$. Indeed, since $2(\widetilde{E}_1+\widetilde{E}_2+\widetilde{E}_3) \sim \pi|_{\widetilde{S}}^* H \sim 2(\overline{E}_1+\overline{E}_2+\overline{E}_3)$, we have $\widetilde{E}_1+\widetilde{E}_2+\widetilde{E}_3 \sim \overline{E}_1+\overline{E}_2+\overline{E}_3$ on the K3 surface $\widetilde{S}$. Let us suppose $\overline{E}_1 \ne \widetilde{E}_i$ for $1\le i\le 3$, so $\overline{E}_1 \cdot \widetilde{E}_i \ge 2$. Then
$4 = \overline{E}_1 \cdot (\overline{E}_1+\overline{E}_2+\overline{E}_3) = \overline{E}_1 \cdot (\widetilde{E}_1+\widetilde{E}_2+\widetilde{E}_3) \ge 6 $,
which is a contradiction.
\end{remark}

\section{SID of curve sections of the known Enriques-Fano threefolds}\label{subsec:sid}

At this stage we can describe
the simple isotropic decompositions of the curve sections of the known Enriques-Fano threefolds. 
We are able to do this for all of them, except for the curve sections of the P-EF 3-fold of genus $13$, for which we find two possible SID. So, the aim of this section is to prove the following theorem.

\begin{theorem}\label{thm:SID}
Let $(W,\mathcal{L})$ be an Enriques-Fano threefold in the list (I)-(XVII) of \S~\ref{sec:list}. Let $S\in \mathcal{L}$ be a general hyperplane section of $W$ and let $H$ be the curve section of $W$ on $S$. Then $H$ has the $\phi$ and the SID described in Table~\ref{tab:riassunto}.
\end{theorem}
\begin{table}[!ht]
\centering
\begin{tabular}{|c|c|c|c|c|} 
\hline
Marking & EF 3-fold &  SID of $H$ & $\phi(H)$ & $(S,H)\in$\\
\hline &&&&\\[-1em]
(\ref{bayle2}) &  $W_{BS}^2$ & $E_1+E_2$ & $1$ & $\mathcal{E}_{2,1}$\\
\hline &&&&\\[-1em]
(\ref{bayle3quadruple}) &  $W_{BS}^3$ & $E_1+E_{1,2}$ & $2$ & $\mathcal{E}_{3,2}$\\
\hline &&&&\\[-1em]
(\ref{bayle3double}) &  $\overline{W}_{BS}^3$ & $2E_1+E_{2}$ & $1$ & $\mathcal{E}_{3,1}$\\
\hline &&&&\\[-1em]
(\ref{bayle4fano}) &  $W_{BS}^4$, $W_F^4$ & $E_1+E_2+E_3$ & $2$ & $\mathcal{E}_{4,2}$\\
\hline &&&&\\[-1em]
(\ref{bayle4}) &  $\overline{W}_{BS}^4$ & $3E_1+E_2$ & $1$ & $\mathcal{E}_{4,1}$\\
\hline &&&&\\[-1em]
(\ref{bayle5birational}) &  $W_{BS}^5$ & $2E_1+E_{1,2}$ & $2$ & $\mathcal{E}_{5,2}^{(I)}$\\
\hline &&&&\\[-1em]
(\ref{bayle5double}) &  $\overline{W}_{BS}^5$ & $2(E_1+E_2)$ & $2$ & $\mathcal{E}_{5,2}^{(II)-}$\\
\hline &&&&\\[-1em]
(\ref{bayle6fano}) &  $W_{BS}^6$, $W_F^6$ & $E_1+E_2+E_{1,2}$ & $3$ & $\mathcal{E}_{6,3}$\\
\hline &&&&\\[-1em]
(\ref{bayle7}) &  $\overline{W}_{BS}^7$ & $3E_1+E_{1,2}$ & $2$ & $\mathcal{E}_{7,2}^{(I)}$\\
 \hline &&&&\\[-1em]
(\ref{bayle7fano}) &  $W_{BS}^7$, $W_F^7$ & $E_1+E_2+E_3+E_4$ & $3$ & $\mathcal{E}_{7,3}$\\
\hline &&&&\\[-1em]
(\ref{bayle8}) &  $W_{BS}^8$ & $2E_1+E_3+E_{1,2}$ & $3$ & $\mathcal{E}_{8,3}$\\
\hline &&&&\\[-1em]
(\ref{bayle9fano}) &  $W_{BS}^{9}$, $W_F^9$ & $2(E_1+E_{1,2})$ & $4$ & $\mathcal{E}_{9,4}^{+}$\\
\hline &&&&\\[-1em]
(\ref{bayle10}) &  $W_{BS}^{10}$ & $2E_1+E_2+E_3+E_4$ & $3$ & $\mathcal{E}_{10,3}^{(I)}$\\
\hline &&&&\\[-1em]
(\ref{bayle13fano}) &  $W_{BS}^{13}$, $W_F^{13}$ & $2(E_1+E_2+E_3)$ & $4$ & $\mathcal{E}_{13,4}^{(II)+}$\\
\hline &&&&\\[-1em]
(\ref{KLM}) &  $W_{KLM}^9$ & $2E_1+2E_2+E_3$ & $3$ & $\mathcal{E}_{9,3}^{(II)}$\\
\hline &&&&\\[-1em]
(\ref{pro13}) &  $W_P^{13}$ & $2(E_1+E_2+E_3)$ & $4$ & $\mathcal{E}_{13,4}^{(II)+}$\\
 &   & or $2(E_1+E_2+E_3)+K_S$ & & or $\mathcal{E}_{13,4}^{(II)-}$\\
\hline &&&&\\[-1em]
(\ref{pro17}) &  $W_P^{17}$ & $4(E_1+E_2)$ & $4$ & $\mathcal{E}_{17,4}^{(IV)+}$\\
\hline
\end{tabular} 
\caption{Simple isotropic decomposition of the curve section class $H$ of an Enriques-Fano threefold $(W,\mathcal{L})$ on a general $S\in \mathcal{L}$.} 
\label{tab:riassunto}   
\end{table}
\begin{proof}
Let us study the known Enriques-Fano threefolds case by case. If $(W,\mathcal{L})$ is a fixed Enriques-Fano threefold of genus $p$, we will denote each time by $\phi_{\mathcal{L}} : W \dashrightarrow \mathbb{P}^{p}$ the rational map defined by $\mathcal{L}$, by $S$ a general element of $\mathcal{L}$, by $H$ a general curve section of $W$ on $S$ satisfying $H^2 = 2p-2$, and by $\phi$ the value $\phi(H)$.
\begin{itemize}
\item[(\ref{bayle2})] $W=W_{BS}^2$. 
The map $\phi_{\mathcal{L}}: W \dashrightarrow \mathbb{P}^2$ is a rational map (see \cite[\S 6.1.6]{Ba94}).
Since $p=2$ and $\phi^2\le 2p-2$, then we have $\phi=1$. So the SID is $H \sim E_1+E_2$ (see Table~\ref{tab:Ep-phi}).

%

\item[(\ref{bayle3quadruple})] $W=W_{BS}^3$.
Since $p=3$ and $\phi^2 \le 2p-2$, then we have $1\le \phi \le 2$.
The map $\phi_{\mathcal{L}}: W \to \mathbb{P}^3$ is a morphism and a quadruple cover (see \cite[\S 6.1.5]{Ba94}). 
This implies $\phi = 2$, because if $\phi = 1$ the map would be a double cover (see \S~\ref{subsec:preliminarSid} (a)). Then the SID is $H \sim E_1+E_{1,2}$ (see Table~\ref{tab:Ep-phi}).

%
%
\item[(\ref{bayle3double})] $W=\overline{W}_{BS}^{3}$.
Since $p=3$ and $\phi^2 \le 2p-2$, then we have $1\le \phi \le 2$.
The map $\phi_{\mathcal{L}}: W \dashrightarrow \mathbb{P}^3$ is a rational map and a double cover (see \cite[\S 6.2.7]{Ba94}). 
We have that $\phi_\mathcal{L} |_{S}$ is hyperelliptic, since it is of degree $2$ onto a plane. This implies $\phi = 1$ (see \S~\ref{subsec:preliminarSid} (a)) and the SID must be $H \sim 2E_1+E_2$ (see Table~\ref{tab:Ep-phi}).

%
%
\item[(\ref{bayle4fano})] $W=W_{BS}^4$. 
Since $p=4$ and $\phi^2 \le 2p-2$, then we have $1\le \phi\le 2$.
The map $\phi_{\mathcal{L}}: W \dashrightarrow \mathbb{P}^4$ is a rational map birational onto the image (see \cite[\S 6.3.3]{Ba94}), which is the Enriques threefold $W_F^4$ of \cite[\S 10]{Fa38}.
Since $S$
is mapped by $\phi_{\mathcal{L}}$ to a general sextic surface of $\mathbb{P}^3$ double along the edges of a tetrahedron, then the SID is $H \sim E_1+E_2+E_3$ (see \cite[\S 5]{CDGK19}) and $\phi = 2$ (see Table~\ref{tab:Ep-phi}).

%

\item[(\ref{bayle4})] $W=\overline{W}_{BS}^4$.
Since $p=4$ and $\phi^2 \le 2p-2$, then we have $1\le \phi\le 2$.
The map $\phi_{\mathcal{L}}: W \dashrightarrow \mathbb{P}^4$ is a rational map and a double cover over its image, which is a quadric cone 
(see \cite[\S 6.6.2]{Ba94}). 
We have that $\phi_\mathcal{L} |_{S}$ is hyperelliptic, since it is of degree $2$ onto a quadric surface of $\mathbb{P}^3$.
Then we have $\phi = 1$ (see \S~\ref{subsec:preliminarSid} (a)) and the SID is $H \sim 3E_1+E_2$ (see Table~\ref{tab:Ep-phi}).

%

\item[(\ref{bayle5birational})] $W=W_{BS}^5$.
Since $p=5$ and $\phi^2 \le 2p-2$, then we have $1\le \phi\le 2$.
The map $\phi_{\mathcal{L}}: W \to \mathbb{P}^5$ is a morphism birational onto its image, 
which has two double planes 
(see \cite[\S 6.2.2]{Ba94}). So we have that $\phi =2$ (see \S~\ref{subsec:preliminarSid} (b)). The SID is $H \sim 2(E_1+E_2)+K_S$ or $H \sim 2E_1+E_{1,2}$ or $H\sim 2(E_1+E_2)$ (see Table~\ref{tab:Ep-phi}). The case $H\sim 2(E_1+E_2)$ is excluded, otherwise the map $\phi_{\mathcal{L}}|_S$ would be superelliptic (see \cite[Theorem 4.7.1]{CoDo89}).
Let us consider now a smooth intersection $B_4 := Q_1\cap Q_2$ of two quadric hypersurfaces of $\mathbb{P}^{5}$ and an elliptic curve $e\subset B_4$ given by the intersection of two hyperplane sections of $B_4$. 
In Bayle's description, an Enriques-Fano threefold $W$ of this type is given by the quotient $\pi : X \to X/\sigma=:W$ of $X:= \operatorname{Bl}_{e}B_4$, that is the blow-up of $B_4$ along the curve $e$, where $\sigma$ is an involution of $X$ with eight fixed points. 
Let us denote the above blow-up by the map $bl : X \to B_4$ and let $E:=bl^{-1}(e)$ be the exceptional divisor. If $h$ denotes the hyperplane class of $\mathbb{P}^5$, then $K_{Q_1}=(K_{\mathbb{P}^5}+Q_1)|_{Q_1}=(-4h)|_{Q_1}$ and $K_{B_4}=(K_{Q_1}+B_4)|_{B_4}=(K_{Q_1}+Q_2|_{Q_1})|_{B_4}=(-2h)|_{B_4}$ (see \cite[p.187]{GH}). 
Furthermore, if $\widetilde{S}$ is the K3-surface $\pi^* S$, then $\pi|_{\widetilde{S}}^*H \sim -K_X |_{\widetilde{S}} = \big(2bl^*(h)-E\big)|_{\widetilde{S}}$. 
Let us see that $E|_{\widetilde{S}}$ is not $2$-divisible. We observe that $\tilde{S}$ is isomorphic to the complete intersection of three quadric hypersurfaces of $\mathbb{P}^5$ 
and that $E|_{\widetilde{S}}$ is a quartic elliptic curve $C$. 
If $E|_{\widetilde{S}}$ were $2$-divisible, we would have a divisor $D$ on $\widetilde{S}$ such that $C\sim 2D$ and $D^2=0$. 
We observe that $-D$ couldn't be effective, otherwise $-2D\sim -C$ would be effective and this is a contradiction; so by Serre Duality we would have $h^2(\mathcal{O}_{\widetilde{S}}(D))=0$. Furthermore, by Riemann-Roch we would obtain
$h^0(\mathcal{O}_{\widetilde{S}}(D)) \ge h^0(\mathcal{O}_{\widetilde{S}}(D))-h^1(\mathcal{O}_{\widetilde{S}}(D)) = 2>0.$
Thus, $D$ would be effective, elliptic (by the adjunction formula)
and with degree $2$, which is a contradiction.
This implies that 
$H$ is not numerically divisible by $2$, 
so the only possible SID is $H \sim 2E_1+E_{1,2}$.


\item[(\ref{bayle5double})] $W=\overline{W}_{BS}^5$.
Since $p=5$ and $\phi^2 \le 2p-2$, then we have $1\le \phi\le 2$.
The map $\phi_{\mathcal{L}}: W \to \mathbb{P}^5$ is a morphism and it is a double cover of the image, 
which is the complete intersection of two quadric hypersurfaces (see \cite[\S 6.1.2]{Ba94}). 
We observe that $\phi_\mathcal{L} |_{S}$ is superelliptic, because it is of degree $2$ onto a quartic surface of $\mathbb{P}^4$.
Hence we have $\phi = 2$, because if $\phi=1$ the map would be hyperelliptic (see \S~\ref{subsec:preliminarSid} (a)). Then the SID is $H\sim 2(E_1+E_2)$ (see Table~\ref{tab:Ep-phi}), since $H$ has to be $2$-divisible in $\operatorname{Pic}(S)$ (see \cite[Theorem 4.7.1]{CoDo89}).

%

\item[(\ref{bayle6fano})] $W=W_{BS}^6$.
Since $p=6$ and $\phi^2 \le 2p-2$, then we have $1\le \phi\le 3$.
The map $\phi_{\mathcal{L}}: W \hookrightarrow \mathbb{P}^6$ is a morphism and it is an isomorphism onto its image (see \cite[\S 6.2.4]{Ba94}). Therefore we have $\phi =3$, otherwise $\phi_{\mathcal{L}}|_{S}$ would not be an isomorphism onto its image (see \S~\ref{subsec:preliminarSid} (c)). Then the SID is $H \sim E_1+E_2+E_{1,2}$ (see Table~\ref{tab:Ep-phi}). We also recall that the F-EF threefold $W_F^6$ of \cite[\S 3]{Fa38} is a limit of $W_{BS}^6$ (see \cite[Main Theorem 2]{Mi99}).



\item[(\ref{bayle7})] $W=\overline{W}_{BS}^7$.
Since $p=7$ and $\phi^2 \le 2p-2$, then we have $1\le \phi\le 3$.
The map $\phi_{\mathcal{L}}: W \to \mathbb{P}^7$ is a morphism, it is birational onto its image but it is not an isomorphism onto its image, since there are points in the image with two preimages (see \cite[\S 6.6.1]{Ba94}). 
Let us explain it better. In Bayle's description, an Enriques-Fano threefold $W$ of this type is given by the quotient $\pi : X \to X/\sigma=:W$ of $X:=\mathbb{P}^1\times S_4$, where $S_{4}\subset \mathbb{P}^{4}$ is a Del Pezzo surface of degree $4$ and $\sigma$ is an involution of $X$ with eight fixed points. In his analysis, Bayle introduces a morphism $\varphi : X \to \mathbb{P}^7$ such that we have the following commutative diagram
\begin{center}
$\begin{tikzcd}
X=\mathbb{P}^1\times S_4 \arrow[d, "\pi"] \arrow[dr, "\varphi"] & \\
W \arrow[r, "\phi_{\mathcal{L}}"]  & \varphi (X) = \phi_{\mathcal{L}} (W) \subset \mathbb{P}^{7}.
\end{tikzcd}$ 
\end{center}
In particular a point $x\in \varphi (X)$ has two preimages in $X$, except in the case in which $x\in \varphi (\left[0:1\right]\times S_4 )\cup \varphi (\left[1:0\right]\times S_4)$: in this case $\varphi^{-1}(x)$ is given by four points of $X$. Since $\pi : X \to W$ has degree $2$, then 
$\phi_{\mathcal{L}}^{-1}(x)$ is given by one point if $x\in \phi_{\mathcal{L}} (W)\setminus \left( \varphi (\left[0:1\right]\times S_4 )\cup \varphi (\left[1:0\right]\times S_4) \right)$, otherwise it is given by two points.
This implies $\phi = 2$ (see \S~\ref{subsec:preliminarSid} (b)). For the SID of $H$ we have a priori two possibilities, namely $H \sim 3E_1+E_{1,2}$ and $H \sim 3E_1+2E_2$ (see Table~\ref{tab:Ep-phi}).

\begin{remark}\label{rem:bayle7cubiccurves}
In the case in which the SID is $H \sim 3E_1+E_{1,2}$, the surface $S$ does not contain elliptic cubic curves. Indeed, we have $\deg E_1 = E_1\cdot H = 2$ and $\deg E_{1,2}= E_{1,2}\cdot H = 6$. Furthermore, let $E$ be an elliptic curve in $S$ such that it is not numerically equivalent to $E_1$, $E_{1,2}$, $2E_1$, $2E_{1,2}$. By \cite[Lemma 2.1]{KL07} we have that $E\cdot E_1 >0$, $E\cdot E_{1,2}>0$ and so $\deg E = E\cdot H \ge 3+1 = 4$. 
\end{remark}

\begin{remark}\label{rem:bayle7cubiccurves2}
In the case in which the SID is $H \sim 3E_1+2E_2$, the surface $S$ contains the following elliptic cubic curves: $E_2$ and $E_2'\sim E_2+K_{S}$.
\end{remark}

It is known that the surface $S_4$ is the image of $\mathbb{P}^2$ via the rational map $\lambda$
defined by the linear system of the plane cubic curves passing through five fixed points $a_1$, $a_2$, $a_3$, $a_4$, $a_5$ in general position. In particular $S_4\cong \operatorname{Bl}_{a_1,a_2,a_3,a_4,a_5}\mathbb{P}^2$, where $bl : \operatorname{Bl}_{a_1,a_2,a_3,a_4,a_5}\mathbb{P}^2 \to \mathbb{P}^2$ is the blow-up of the plane at these five points. 
Let $\ell$ be the strict transform of a general line of $\mathbb{P}^{2}$ and let us consider the exceptional divisors $e_i := bl^{-1}(a_i)$ of $bl: S_4 \to \mathbb{P}^2$, for $1\le i \le 5$. 
Let us take 
the K3-surface $\widetilde{S} :=\pi^* S$. Then we have that
\begin{center}
$\pi|_{\widetilde{S}}^*H \sim -K_X |_{\widetilde{S}} \sim (2p\times S_4 + \mathbb{P}^1\times (-K_{S_4}))|_{\widetilde{S}} \sim$

$\sim \left(2p\times S_4 + \mathbb{P}^1 \times (3\ell -e_1-e_2-e_3-e_4-e_5)\right)|_{\widetilde{S}}=$

$=2p\times S_4 |_{\widetilde{S}} + \mathbb{P}^1\times (\ell -e_5)|_{\widetilde{S}}+\mathbb{P}^1\times (2\ell -e_1-e_2-e_3-e_4)|_{\widetilde{S}}.$
\end{center}
By setting $\overline{E}_1:=\mathbb{P}^1\times (\ell -e_5)|_{\widetilde{S}}$, $\overline{E}_2:=\mathbb{P}^1\times (2\ell -e_1-e_2-e_3-e_4)|_{\widetilde{S}}$ and $\overline{E}_3:=p\times S_4 |_{\widetilde{S}}$, we have 
$\pi|_{\widetilde{S}}^*H\sim \overline{E}_1+\overline{E}_2+2\overline{E}_3$, 
where $\overline{E}_1^2 = \overline{E}_2^2 =\overline{E}_3^2 = 0$, $\overline{E}_1\cdot \overline{E}_2 =4$ and $\overline{E}_1\cdot \overline{E}_3 =\overline{E}_2\cdot \overline{E}_3 =2$.
Furthermore, by the adjunction formula, we have that $K_{\overline{E}_i}=0$ and $p_g(\overline{E}_i)=1$, for $1\le i\le 3$. Let us suppose that there exists an elliptic cubic curve $E$ on $S$ and let us define $\overline{E}:=\pi^{-1}(E)$. Since $E\cdot H = 3$ on $S$, then $\overline{E}\cdot \pi|_{\widetilde{S}}^*H = 6$ on $\widetilde{S}$. Obviously, we have that $\overline{E}$ is not linearly equivalent to $\overline{E}_1$, $\overline{E}_2$, $\overline{E}_3$, because $\overline{E}_1 \cdot \pi|_{\widetilde{S}}^*H = \overline{E}_2 \cdot \pi|_{\widetilde{S}}^*H = 8\ne 6$ and $\overline{E}_3 \cdot \pi|_{\widetilde{S}}^*H = 4\ne 6$.  
Since two elliptic curves on a K3 surface, which are not linearly equivalent, intersect at least in two points, then $\overline{E} \cdot \pi|_{\widetilde{S}}^*H \ge 2+2+2\cdot 2 = 8>6$, which is a contradiction. Hence $S$ does not contain elliptic cubic curves and so, by Remarks~\ref{rem:bayle7cubiccurves}, ~\ref{rem:bayle7cubiccurves2}, the SID is $H \sim 3E_1+E_{1,2}$ with $\phi = 2.$

%

\item[(\ref{bayle7fano})] $W=W_{BS}^7$.
Since $p=7$ and $\phi^2 \le 2p-2$, then we have $1\le \phi\le 3$.
The map $\phi_{\mathcal{L}}: W \hookrightarrow \mathbb{P}^7$ is a morphism and it is an isomorphism onto its image (see \cite[\S 6.4.1]{Ba94}). This implies $\phi = 3$ (see \S~\ref{subsec:preliminarSid} (c)), which yields the SID $H \sim E_1+E_2+E_3+E_4$ (see Table~\ref{tab:Ep-phi}). See also \cite[Lemma 4.6]{CDGK20}, where these threefolds are obtained via a \textit{projection} technique from (\ref{bayle13fano}). The F-EF threefold $W_F^7$ of \cite[\S 4]{Fa38} is a limit of $W_{BS}^7$ (see \cite[Main Theorem 2]{Mi99}).


\item[(\ref{bayle8})] $W=W_{BS}^8$.
Since $p=8$ and $\phi^2 \le 2p-2$, then we have $1\le \phi \le 3$.
The map $\phi_{\mathcal{L}}: W \hookrightarrow \mathbb{P}^8$ is a morphism and it is an isomorphism onto its image (see \cite[\S 6.4.2]{Ba94}). This implies $\phi = 3$ (see \S~\ref{subsec:preliminarSid} (c)), which yields 
$H \sim 2E_1+E_3+E_{1,2}$ (see Table~\ref{tab:Ep-phi}).



\item[(\ref{bayle9fano})] $W=W_{BS}^9$. 
Since $p=9$ and $\phi^2 \le 2p-2$, then we have $1\le \phi \le 4$.
The map $\phi_{\mathcal{L}}: W \hookrightarrow \mathbb{P}^9$ is a morphism and it is an isomorphism onto its image (see \cite[\S 6.1.4]{Ba94}). Therefore one has $3\le \phi \le 4$ (see \S~\ref{subsec:preliminarSid} (c)). In Bayle's description, an Enriques-Fano threefold $W$ of this type is given by the quotient $\pi : X \to X/\sigma=:W$ of the complete intersection $X$ of two quadric hypersurfaces of $\mathbb{P}^5$, where $\sigma$ is an involution of $X$ with eight fixed points. This implies that 
$H$ is numerically divisible by $2$: indeed, if $\widetilde{S}$ is the K3-surface $\pi^* S$, then $\pi|_{\widetilde{S}}^*H \sim -K_X |_{\widetilde{S}}$ where $-K_X$ is a quadric section of $X$. So we have $\phi = 4$ and $H\sim 2(E_1+E_{1,2})$ or $H\sim 2(E_1+E_{1,2})+K_S$ (see Table~\ref{tab:Ep-phi}).
We recall that the F-EF threefold $W_F^9$ of \cite[\S 7]{Fa38} is a limit of $W_{BS}^9$ (see \cite[Main Theorem 2]{Mi99}). 
Furthermore, Fano's description shows that $H$ is $2$-divisible in $\operatorname{Pic}(S)$ (see Theorem~\ref{thm:divisibilityK}). 
This implies that the only possible SID is $H \sim 2(E_1+E_{1,2})$.


\item[(\ref{bayle10})] $W=W_{BS}^{10}$. Since $p=10$ and $\phi^2 \le 2p-2$, then we have $1\le \phi \le 4$. The map $\phi_{\mathcal{L}}: X \hookrightarrow \mathbb{P}^{10}$ is a morphism and it is an isomorphism onto its image (see \cite[\S 6.5.1]{Ba94}). Therefore one has $3\le \phi \le 4$ (see \S~\ref{subsec:preliminarSid} (c)). The possible cases of the SID of $H$ are 
$H\sim 2E_1+E_2+E_3+E_4$, $H\sim 3(E_1+E_2)$ and $H\sim 2E_{1,2}+E_1+E_2$ (see Table~\ref{tab:Ep-phi}).
We recall that an Enriques-Fano threefold $W$ of this type is given by the quotient $\pi : X \to X/\sigma=:W$ of $X:=\mathbb{P}^1\times S_6$, where $S_6$ is a smooth Del Pezzo surface of degree $6$ in $\mathbb{P}^{6}$ and $\sigma$ is an involution of $X$ with eight fixed points.
It is known that the surface $S_6$ is the image of $\mathbb{P}^2$ via the rational map 
defined by the linear system of the plane cubic curves passing through three fixed points $a_1$, $a_2$, $a_3$ in general position. In particular $S_6\cong \operatorname{Bl}_{a_1,a_2,a_3}\mathbb{P}^2$, where $bl : \operatorname{Bl}_{a_1,a_2,a_3}\mathbb{P}^2 \to \mathbb{P}^2$ is the blow-up of the plane at these three points. 
Let $\ell$ be the strict transform of a general line of $\mathbb{P}^{2}$ and let us consider the exceptional divisors $e_i = bl^{-1}(a_i)$ of $bl: S_6 \to \mathbb{P}^2$, for $1\le i\le 3$. 
Let us take 
the K3-surface $\widetilde{S} :=\pi^* S$. Then we have that
\begin{center}
$\pi|_{\widetilde{S}}^*H \sim -K_X |_{\widetilde{S}} \sim (2p\times S_6 + \mathbb{P}^1\times (-K_{S_6}))|_{\widetilde{S}}
= \left(2p\times S_6 + \mathbb{P}^1 \times (3\ell -e_1-e_2-e_3)\right)|_{\widetilde{S}}=$

$=2p\times S_6 |_{\widetilde{S}} + \mathbb{P}^1\times (\ell -e_1)|_{\widetilde{S}}+\mathbb{P}^1\times (\ell -e_2)|_{\widetilde{S}}+\mathbb{P}^1\times (\ell -e_3)|_{\widetilde{S}}.$
\end{center}
By setting $\overline{E}_1:=p\times S_6 |_{\widetilde{S}}$ and $\overline{E_i}:=\mathbb{P}^1\times (\ell -e_i)|_{\widetilde{S}}$, for $2\le i\le 4$, we have 
$\pi|_{\widetilde{S}}^*H\sim 2\overline{E}_1+\overline{E}_2+\overline{E}_3+\overline{E}_4$, 
where $\overline{E}_i^2 = 0$ and $\overline{E}_i\cdot \overline{E}_j =2$, for $1\le i<j \le 4$.
Furthermore, by the adjunction formula, we have that $K_{\overline{E}_i}=0$ and $p_g(\overline{E}_i)=1$, for $1\le i\le 4$. We will now prove that the SID is $H\sim 2E_1+E_2+E_3+E_4$. If the SID were $H\sim 3(E_1+E_2)$, then $H$ would be $3$-divisible and therefore also $\pi|_{\widetilde{S}}^*H$. But this does not happen because $\pi|_{\widetilde{S}}^*H\cdot \overline{E}_2 = 8$ is not divisible by $3$. Now suppose $H\sim 2E_{1,2}+E_{1}+E_{2}$. By setting $\widetilde{E}_{1,2}:=\pi|_{\widetilde{S}}^*E_{1,2}$, $\widetilde{E}_{1}:=\pi|_{\widetilde{S}}^*E_{1}$ and $\widetilde{E}_{2}:=\pi|_{\widetilde{S}}^*E_{2}$, we have $\widetilde{E}_{1,2}\cdot \widetilde{E}_{1}=\widetilde{E}_{1,2}\cdot \widetilde{E}_{2}=4$ and $\widetilde{E}_{1}\cdot \widetilde{E}_{2}=2$. Hence $\widetilde{E}_{1,2}\cdot\pi|_{\widetilde{S}}^*H = 8$ and $\widetilde{E}_{1}\cdot\pi|_{\widetilde{S}}^*H = \widetilde{E}_{2}\cdot\pi|_{\widetilde{S}}^*H = 10$. Let $D$ be any elliptic curve on $S$ such that $D^2=0$ and that is not linearly equivalent to $\widetilde{E}_{1,2}$, $\widetilde{E}_1$, $\widetilde{E}_2$: then $D\cdot \pi|_{\widetilde{S}}^*H \ge 2\cdot 2+2+2 = 8$,
since two elliptic curves on a K3 surface, which are not linearly equivalent, intersect at least in two points.
But if we took $D=\overline{E}_1$, we would obtain $\overline{E}_1\cdot \pi|_{\widetilde{S}}^*H = 6<8$, which is a contradiction. Then it must be $H\sim 2E_1+E_2+E_3+E_4$ with $\phi=3$.



\item[(\ref{bayle13fano})] $W=W_{BS}^{13}$.
Since $p=13$ and $\phi^2 \le 2p-2$, then we have $1\le \phi \le 4$.
The map $\phi_{\mathcal{L}}: W \hookrightarrow \mathbb{P}^{13}$ is a morphism and it is an isomorphism onto its image (see \cite[\S 6.3.2]{Ba94}). Therefore one has $3\le \phi \le 4$ (see \S~\ref{subsec:preliminarSid} (c)). According to Bayle, an Enriques-Fano threefold $W$ of this type is given by the quotient $\pi : X \to X/\sigma=:W$ of $X:=\mathbb{P}^{1} \times \mathbb{P}^{1} \times \mathbb{P}^{1}$ under an involution $\sigma$ of $X$ with eight fixed points. So Bayle's description implies that 
$H$ is numerically divisible by $2$: indeed, if $\widetilde{S}$ is the K3-surface $\pi^* S$, then $\pi|_{\widetilde{S}}^*H \sim -K_X |_{\widetilde{S}} \sim (2,2,2)|_{\widetilde{S}}.$ 
We recall that the F-EF threefold $W_F^{13}$ of \cite[\S 8]{Fa38} is a limit of $W_{BS}^{13}$ (see \cite[Main Theorem 2]{Mi99}). 
Furthermore, Fano's description shows that $H$ is $2$-divisible in $\operatorname{Pic}(S)$ (see Theorem~\ref{thm:divisibilitySigma}). 
So the only possible case is $\phi = 4$ and 
$H \sim 2(E_1+E_2+E_3)$ (see Table~\ref{tab:Ep-phi}).



\item[(\ref{KLM})] $W=W_{KLM}^9$.
These threefolds are obtained by \textit{projection} of the threefolds in (\ref{bayle13fano}) from, say, the curve $E_3$ (see \cite{KLM11}). Then $H \sim 2(E_1+E_2)+E_3$ and $\phi = 3$. 

\item[(\ref{pro13})] $W=W_P^{13}$.
By Theorem~\ref{thm:SIDpro13} we have $\phi = 4$ and the following two possibile SID: $H\sim 2(E_1+E_2+E_3)$ or $H\sim 2(E_1+E_2+E_3)+K_S$.

\item[(\ref{pro17})] $W=W_P^{17}$ - see \cite[\S 3.3]{Pro07}. 
By \cite[Proposition 4.7]{CDGK20} the SID is $H \sim 4(E_1+E_2)$ and $\phi=4$.


\end{itemize}
\end{proof}

%

\subsection*{Open question}
It remains to understand which case really occurs concerning the SID of the curve sections of the P-EF 3-fold of genus 13.

\end{document}